\documentclass[reqno,10pt,a4paper,final]{amsart}
\usepackage{amssymb,latexsym,graphicx,enumerate}
\usepackage[pagebackref]{hyperref}
\usepackage[a4paper,margin=3cm]{geometry}
\usepackage{showkeys} 

\numberwithin{equation}{section}%
\numberwithin{figure}{section}

\title{Spectrum of Markov generators on sparse random graphs}
\date{Preprint, to appear in Communications on Pure and Applied Mathematics (CPAM, 2014)}%

\author{Charles Bordenave}%
\address[Ch.~Bordenave]{CNRS \& Institut de Math\'ematiques de Toulouse, Universit\'e de Toulouse, France}
\email{charles.bordenave(at)math.univ-toulouse.fr}%
\urladdr{http://www.math.univ-toulouse.fr/~bordenave/}

\author{Pietro Caputo}%
\address[P.~Caputo]{Dipartimento di Matematica, Universit\`a Roma Tre, Italy}
\email{caputo(at)mat.uniroma3.it}%
\urladdr{http://www.mat.uniroma3.it/users/caputo/}

\author{Djalil Chafa\"\i}%
\address[D.~Chafa{\"{\i}}]{LAMA Universit\'e Paris-Est Marne-la-Vall\'ee \&
  CEREMADE Universi\'e Paris-Dauphine \& Institut Universitaire de France, France}
\email{djalil(at)chafai.net}%
\urladdr{http://djalil.chafai.net/}

\newtheorem{theorem}{Theorem}[section]
\newtheorem{lemma}[theorem]{Lemma}
\newtheorem{corollary}[theorem]{Corollary}
\newtheorem{proposition}[theorem]{Proposition}
\theoremstyle{definition}

\theoremstyle{remark}
\newtheorem{remark}[theorem]{Remark}

\newcommand{\dC}{\mathbb{C}}
\newcommand{\dE}{\mathbb{E}}
\newcommand{\dH}{\mathbb{H}}

\newcommand{\dN}{\mathbb{N}}
\newcommand{\dP}{\mathbb{P}}
\newcommand{\dR}{\mathbb{R}}
\newcommand{\dS}{\mathbb{S}}

\newcommand{\dZ}{\mathbb{Z}}

\newcommand{\cA}{\mathcal{A}}
\newcommand{\cC}{\mathcal{C}}\newcommand{\cD}{\mathcal{D}}
\newcommand{\cF}{\mathcal{F}}

\newcommand{\cL}{\mathcal{L}}
\newcommand{\cM}{\mathcal{M}}\newcommand{\cN}{\mathcal{N}}
\newcommand{\cO}{\mathcal{O}}

\newcommand{\al}{\alpha}

\newcommand{\de}{\delta}
\newcommand{\ga}{\gamma}

\newcommand{\la}{\lambda}
\newcommand{\La}{\Lambda}

\newcommand{\Om}{\Omega}

\newcommand{\Si}{\Sigma}
\newcommand{\si}{\sigma}

\newcommand{\te}{\theta}

\newcommand{\veps}{\varepsilon}

\newcommand{\bx}{{\mathbf{x}}}

\newcommand{\by}{{\mathbf{y}}}

\newcommand{\FLOOR}[1]{{{\lfloor#1\rfloor}}} %
\newcommand{\CEIL}[1]{{{\lceil#1\rceil}}} %
\newcommand{\ABS}[1]{{{\left| #1 \right|}}} 
\newcommand{\BRA}[1]{{{\left\{#1\right\}}}} 
\newcommand{\DOT}[1]{{{\left<#1\right>}}} 
\newcommand{\ANG}[1]{{{\langle#1\rangle}}} 
\newcommand{\NRM}[1]{{{\left\| #1\right\|}}} 
\newcommand{\PAR}[1]{{{\left(#1\right)}}} 
\newcommand{\pd}{{\partial}} 
\newcommand{\SBRA}[1]{{{\left[#1\right]}}} 
\newcommand{\TV}[1]{\NRM{#1}_\textsc{TV}} 
\newcommand{\BV}[1]{\NRM{#1}_\textsc{BV}} 
\newcommand{\RANK}{\mathrm{rank}}
\newcommand{\BIP}{\mathrm{bip}}
\newcommand{\DIM}{\mathrm{dim}}
\newcommand{\SUPP}{\mathrm{supp}}
\newcommand{\DIST}{\mathrm{dist}}
\newcommand{\SPAN}{\mathrm{span}}
\newcommand{\COMP}{\mathrm{Comp}}
\newcommand{\INCOMP}{\mathrm{Incomp}}
\newcommand{\SPARSE}{\mathrm{Sparse}}
\newcommand{\TR}{\mathrm{Tr}}
\newcommand{\DIAG}{\mathrm{diag}}
\newcommand{\VAR}{\mathrm{Var}}
\newcommand{\COV}{\mathrm{Cov}}
\newcommand{\IND}{\mathbf{1}}

\newcommand{\Ul}{\underline}

\newcommand{\Wt}{\widetilde}

\newcommand{\weak}{\rightsquigarrow}
\renewcommand{\Im}{\mathfrak{Im}}
\renewcommand{\Re}{\mathfrak{Re}}

\begin{document}                        

\begin{abstract}
  We investigate the spectrum of the infinitesimal generator of the continuous
  time random walk on a randomly weighted oriented graph. This is the
  non-Hermitian random $n\times n$ matrix $L$ defined by $L_{jk}=X_{jk}$ if
  $k\neq j$ and $L_{jj}=-\sum_{k\neq j}L_{jk}$, where $(X_{jk})_{j\neq k}$ are
  i.i.d.\ random weights. Under mild assumptions on the law of the weights, we
  establish convergence as $n\to\infty$ of the empirical spectral distribution
  of $L$ after centering and rescaling. In particular, our assumptions include
  sparse random graphs such as the oriented Erd\H{o}s-R\'enyi graph where each
  edge is present independently with probability $p(n)\to0$ as long as
  $np(n)\gg (\log(n))^6$. The limiting distribution is characterized as an
  additive Gaussian deformation of the standard circular law. In free
  probability terms, this coincides with the Brown measure of the free sum of
  the circular element and a normal operator with Gaussian spectral measure.
  The density of the limiting distribution is analyzed using a subordination
  formula. Furthermore, we study the convergence of the invariant measure of
  $L$ to the uniform distribution and establish estimates on the extremal
  eigenvalues of $L$.
\end{abstract}

\maketitle   

{\small\tableofcontents}

\section{Introduction}
\label{se:intro}
For each integer $n\geq 1$ let $X=(X_{ij})_{1\leq i,j\leq n}$ be the random
matrix whose entries are i.i.d.\ copies of a complex valued random variable
$\bx$ with variance $\si^2$. The circular law theorem (see e.g.\ the survey
papers \cite{MR2507275,bordenave-chafai-changchun}) asserts that the empirical
spectral distribution of $X$ - after centering and rescaling by $\si\sqrt{n}$
- converges weakly to the uniform distribution on the unit disc of $\dC$. The
sparse regime is obtained by allowing the law of $\bx$ to depend on $n$ with a
variance satisfying $\si^2(n)\to 0$ as $n\to\infty$. As an example, if
$\bx=\epsilon(n)$ is a Bernoulli random variable with parameter $p(n)$, then
$X$ is the adjacency matrix of the oriented Erd\H{o}s-R\'enyi random graph
where each edge is present independently with probability $p(n)$. In this
example $\si^2(n)=p(n)(1-p(n))\sim p(n)$ when $p(n)\to0$. It is expected that
the circular law continues to hold in the sparse regime, as long as
$n\si^2(n)\to\infty$. Results in this direction have been recently established
in \cite{MR2409368,gotze-tikhomirov-new,MR2977992}, where the convergence is
proved under some extra assumptions including that $n^{1-\veps}\si^2\to\infty$
for some $\veps>0$.

In this paper, we consider instead random matrices of the form 
\begin{equation}\label{eq:defL}
L=X-D
\end{equation}
where $X$ is a matrix with i.i.d.\ entries as above, and $D$ is the diagonal
matrix obtained from the row sums of $X$, i.e.\ for $i=1,\dots, n$,
\[
D_{ii}=\sum_{k=1}^nX_{ik}.
\]
If $X$ is interpreted as the adjacency matrix of a weighted oriented graph,
then $L$ is the associated Laplacian matrix, with zero row sums. In
particular, if the weights $X_{ij}$ take values in $[0,\infty)$, then $L$ is
the infinitesimal generator of the continuous time random walk on that graph,
and properties of the spectrum of $L$ can be used to study its long-time
behavior. Clearly, $L$ has non independent entries but independent rows. A
related model is obtained by considering the stochastic matrix $P=D^{-1}X$.
The circular law for the latter model in the non sparse regime has been
studied in \cite{MR2892961}. Here, we investigate the behavior of the spectrum
of $L$ both in the non sparse and the sparse regime. We are mostly concerned
with three issues:
\begin{enumerate}[(1)]
\item convergence of the empirical spectral distribution; 
\item properties of the limiting distribution; 
\item invariant measure and extremal eigenvalues of $L$. 
\end{enumerate}
As in the case of the circular law, the main challenge in establishing point
(1) is the estimate on the smallest singular value of deterministic shifts of
$L$. This is carried out by combining the method of Rudelson and Vershynin
\cite{MR2407948}, and G\"otze and Tikhomirov \cite{gotze-tikhomirov-new},
together with some new arguments needed to handle the non independence of the
entries of $L$ and the possibility of zero singular values - for instance, $L$
itself is not invertible since all its rows sum to zero. As for point (2) the
analysis of the resolvent is combined with free probabilistic arguments to
characterize the limiting distribution as the Brown measure of the free sum of
the circular element and an unbounded normal operator with Gaussian spectral
measure. Further properties of this distribution are obtained by using a
subordination formula. This result can be interpreted as the asymptotic
independence of $X$ and $D$. The Hermitian counterpart of these facts has been
discussed by Bryc, Dembo and Jiang in\ \cite{MR2206341}, who showed that if
$X$ is an i.i.d.\ Hermitian matrix, then the limiting spectral distribution of
$L$ - after centering and rescaling - is the free convolution of the
semi-circle law with a Gaussian law. Finally, in point (3) we estimate
the 
total variation distance between the invariant measure of $L$ and the uniform
distribution. This analysis is based on perturbative arguments similar to
those recently used to analyze bounded rank perturbations of matrices with
i.i.d.\ entries \cite{MR1062321,MR1284550,MR2782201,MR3010398}. Further
perturbative reasoning is used to give bounds on the spectral radius and on
the spectral gap of $L$.

Before stating our main results, we introduce the notation to be used. 
If $A$ is an $n\times n$ 
matrix, we denote by
$\la_1(A),\ldots,\la_n(A)$ its eigenvalues, i.e.\ the roots in $\dC$
of its characteristic polynomial. We label them in such a way that
$|\la_1(A)|\geq\cdots\geq|\la_n(A)|$. 
We denote by $s_1(A),\ldots,s_n(A)$ the singular values of $A$, i.e.\ the
eigenvalues of the Hermitian positive semidefinite matrix $| A |=
\sqrt{A^*A}$, labeled so that $s_1(A)\geq\cdots \geq s_n(A) \geq 0$. The
\emph{operator norm} of $A$ is $\| A \|= s_1(A)$ while the \emph{spectral
  radius} is $|\la_1(A)|$. We define the discrete probability measures
\[
\mu_A=\frac{1}{n}\sum_{k=1}^n\delta_{\la_k(A)}
\quad\text{and}\quad
\nu_A=\mu_{|A|}%
=\frac{1}{n}\sum_{k=1}^n\delta_{s_k(A)}.
\]
We denote ``$\weak$'' the weak convergence of measures against bounded
continuous functions on $\dR$ or on $\dC$. If $\mu$ and ${(\mu_n)}$ are random
finite measures on $\dR$ or on $\dC$, we say that $\mu_n\weak\mu$ \emph{in
  probability} when for every bounded continuous function $f$ and for every
$\veps>0$,
$\lim_{n\to\infty}\dP(\ABS{\int\!f\,d\mu_n-\int\!f\,d\mu}>\veps)=0$. In this
paper, all the random variables are defined on a common probability space
$(\Om,\cA,\dP)$. For each integer $n\geq1$, let $\bx=\bx(n)$ denote a complex
valued random variable with law $\cL=\cL(n)$ possibly dependent on $n$, with
variance
\[
\si^2(n) = \VAR(\bx)
=\dE(|\bx-\dE \bx|^2)
=\VAR(\Re \,\bx)+\VAR(\Im\,\bx),
\]
and mean $m(n)=\dE\bx$.
Throughout the paper, it is always assumed that $\bx$ satisfies
\begin{equation}\label{eq:hyp_sig}
  \lim_{n \to \infty} n\si^2(n)= \infty,\quad \sup_{n}\si^2(n)<\infty,
\end{equation}
and the Lindeberg type condition:
\begin{equation} \label{eq:hyp_lind}
  \text{For all $\veps >0$},\quad %
  \lim_{n \to \infty} \si(n)^{-2}\dE \SBRA{%
    |\bx-m(n)|^2%
    \IND_{\BRA{|\bx-m(n) |^2 \geq \veps n\si^2(n) }} %
  } = 0.
\end{equation}
It is also assumed that the normalized covariance matrix converges:
\begin{equation}\label{covK}
  \lim_{n \to \infty}
  \si(n)^{-2}      
  \COV\PAR{\Re(\bx),\Im(\bx)}
  = K
\end{equation}
for some $2\times 2$ matrix $K$. This allows the real and imaginary parts of
$\bx$ to be correlated. Note that the matrix $K$ has unit trace by
construction.

The two main examples we have in mind are:
\begin{enumerate}[\bfseries A)]
\item non sparse case: $\bx$ has law $\cL$ independent of $n$ with finite
  positive variance;
\item sparse case: 
  \begin{equation}\label{epsy}
    \bx=\epsilon(n)\by, 
  \end{equation}
  with $\by$ a bounded random variable with law independent of $n$ and
  $\epsilon(n)$ an independent Bernoulli variable with parameter $p(n)$
  satisfying $p(n)\to 0$ and $np(n)\to\infty$.
\end{enumerate}
These cases are referred to as model A and model B in the sequel. It is
immediate to check that either case satisfies the assumptions
\eqref{eq:hyp_sig}, \eqref{eq:hyp_lind} and \eqref{covK}. We will sometimes
omit the script $n$ in $\bx(n)$, $\cL(n)$, $m(n)$, $\si(n)$, etc. The
matrix $L$ is defined by \eqref{eq:defL}, and we consider the rescaled matrix
\begin{equation}\label{eq:defM} 
  M %
  =\frac{L+nmI}{\si\sqrt{n}} %
  =\frac{X}{\si\sqrt{n}}-\frac{D-nmI}{\si\sqrt{n}}.
\end{equation}
By the central limit theorem, the distribution of
$(D_{ii}-nm)/(\si\sqrt{n})$ converges to the Gaussian law with mean $0$ and
covariance $K$. Combined with the circular law for $X / (\si \sqrt{n})$, this
suggests the interpretation of the spectral distribution of $M$, in the limit
$n\to\infty$, as an additive Gaussian deformation of the circular law.

Define $\dC_+=\{z\in\dC:\Im(z)>0\}$. If $\nu$ is a probability measure on
$\dR$ then its Cauchy-Stieltjes transform is the analytic function
$S_\nu:\dC_+ \to \dC_+$ given for any $z\in\dC_+$ by
\begin{equation}\label{eq:agnuz}
  S_{\nu}(z)=\int\!\frac{1}{t-z}\,d\nu(t).
\end{equation}
We denote by $\check\nu$ the symmetrization of $\nu$, defined for any Borel
set $A$ of $\dR$ by
\begin{equation}\label{eq:bgnuz}
  \check \nu(A) = \frac{\nu(A)+\nu(-A)}{2}.
\end{equation}
If $\nu$ is supported in $\dR_+$ then $\nu$ is characterized by its
symmetrization $\check\nu$. In the sequel, $G$ is a Gaussian random variable
on $\dR^2\cong\dC$ with law $\cN(0,K)$ i.e.\ mean $0$ and covariance matrix
$K$. This law has a Lebesgue density on $\dR^2$ if and only if $K$ is
invertible, given by
$(2\pi\sqrt{\det(K)})^{-1}\exp(-\frac{1}{2}\DOT{K^{-1}\cdot,\cdot})$.

\subsection{Convergence results}

We begin with the singular values of shifts of the matrix $M$, a useful proxy
to the eigenvalues.

\begin{theorem}[Singular values]\label{th:singvals}
  For every $z\in\dC$, there exists a probability measure $\nu_z$ on $\dR_+$
  which depends only on $z$ and $K$ such that with probability one,
  \[
  \nu_{M-zI} %
  \underset{n\to\infty}{\weak}%
  \nu_z.
  \]
  Moreover, the limiting law $\nu_z$ is characterized as follows: $\check
  \nu_z$ is the unique symmetric probability measure on $\dR$ with
  Cauchy-Stieltjes transform satisfying, for every $\eta\in\dC_+$,
 \begin{equation} \label{gnut}
   S_{\check \nu_z}(\eta) %
   =\dE\PAR{\frac{S_{\check \nu_z}(\eta)+\eta } %
   {|G-z|^2-(\eta+ S_{\check \nu_z}(\eta))^2}}.
 \end{equation} 
\end{theorem}

The next result concerns the eigenvalues of $M$. On top of our running
assumptions \eqref{eq:hyp_sig}, \eqref{eq:hyp_lind} and \eqref{covK}, here we
need to assume further:
 \begin{enumerate}[(i)]
  \item variance growth:
    \begin{equation} \label{eq:hyp_sig2} 
      \lim_{n \to \infty} \frac{n\si^2(n)}{(\log(n))^6} = +\infty;
    \end{equation}
  \item 
  tightness conditions: 
  \begin{equation} \label{eq:hyp_tightness1} %
    \lim_{t\to\infty} \inf_{n \geq 1}%
    \frac{\dE\SBRA{|\bx|^2\IND_{\{ \frac{1}{t}<|\bx|< t\}}}}%
    {\dE\SBRA{|\bx|^2}} =1 %
    \quad\text{and}\quad %
    \sup_{n \geq 1}\frac{\dE\SBRA{|\bx|^2}}{\si^2(n)}<\infty;
   \end{equation}
 \item the set $\La$ of accumulation points of $(\sqrt{n}\,m(n)/\si(n))_{n
     \geq 1}$ has zero Lebesgue measure in $\dC$.
  \end{enumerate}
  It is not hard to check that assumptions (i),(ii),(iii) are all satisfied by
  model A and model B, provided in B we require that $p(n)\gg (\log(n))^6/n$.

\begin{theorem}[Eigenvalues]\label{th:eigenvals}
  Assume  that (i),(ii),(iii) above hold. Let $\mu$ be the probability
  measure on $\dC$ defined by
  \[
  \mu=-\frac{1}{2\pi}\Delta U
  \quad\text{with}\quad
  U(z):=-\int_0^\infty\!\log(t)\,d\nu_z(t),
  \]
  where the Laplacian $\Delta=\partial_x^2+\partial_y^2$ is taken in the sense
  of Schwartz-Sobolev distributions in the space $\cD'(\dR^2)$, and where
  $\nu_z$ is as in theorem \ref{th:singvals}. Then, in probability,
  \[
  \mu_{M} %
  \underset{n\to\infty}{\weak}%
  \mu.
  \]
  \end{theorem}

\subsection{Limiting distribution}

The limiting distribution in theorem \ref{th:eigenvals} is independent of the
mean $m$ of the law $\cL$. This is rather natural since shifting the entries
produces a deterministic rank one perturbation. As in other known
circumstances, a rank one additive perturbation produces essentially a single
outlier, and therefore does not affect the limiting spectral distribution, see
e.g.\ \cite{MR1062321,MR1284550,djalil-nccl,MR3010398}. To obtain further
properties of the limiting distribution, we turn to {\em free probability}. 

We refer to \cite{AGZ} and references therein for an introduction to the basic concepts of free probability. 
Recall that a non-commutative probability space is a pair $(\cM, \tau)$ where $\cM$
is a von Neumann algebra and $\tau$ is a normal,
faithful, tracial state on $\cM$. Elements of $\cM$ are bounded linear operators on a Hilbert space. In the present work, we need to deal with
possibly unbounded operators in order to interpret the large $n$ limit of $\si^{-1}
n^{-1/2}(D-nmI)$. To this end, one extends $\cM$ to the so-called {\em affiliated} algebra $\bar \cM
\supset \cM$.
Following Brown \cite{MR866489} and
Haagerup and Schultz \cite{MR2339369}, one can associate to every element $a\in\bar \cM$ a
probability measure $\mu_a$ on $\dC$, called the {\em Brown spectral measure}.
If $a$ is normal, i.e.\ if $a^*\!a=a a^*$, then the Brown measure coincides
with the usual spectral measure of a normal operator on a Hilbert space. The usual notion of $\star$-free operators still makes
sense in $(\bar\cM,\tau)$ even if the elements of $\bar \cM$ are not
necessarily bounded. We refer to section \ref{subsec:Brown} below for precise
definitions in our setting and to \cite{MR2339369} for a complete treatment.
We use the standard notation $|a|=\sqrt{a^*\!a}$ for the square root of the non
negative self-adjoint operator $a^*\!a$.

\begin{theorem}[Free probability interpretation of limiting
  laws]\label{th:naturemu}
  Let $c$ and $g$ be $\star$-free operators in $(\bar \cM, \tau)$, with $c$
  circular, and $g$ normal operator\footnote{Normal means $gg^*=g^*g$, a property which
    has nothing to do with Gaussianity. However, and coincidentally, it turns
    out that the spectral measure of $g$ is additionally assumed Gaussian
    later on!} with spectral measure equal to $\cN(0,K)$. Then, if $\nu_z$ and
  $\mu$ are as in theorems \ref{th:singvals}-\ref{th:eigenvals}, we have
  \[
  \nu_z = \mu_{|c + g - z|} \quad \text{and} \quad \mu = \mu_{c+g}.
  \] 
\end{theorem}

Having identified the limit law $\mu$, we obtain some additional information
on it.

\begin{theorem}[Properties of the limiting measure]\label{th:propmu}
  Let $c$ and $g$ be as in theorem \ref{th:naturemu}. The support of the Brown
  measure $\mu_{c+g}$ of $c + g$ is given by
  \[
  \SUPP (\mu_{c+g}) %
  = \BRA{z \in \dC : \dE\PAR{\frac{1}{|G-z|^2}} \geq 1}.
  \]
  There exists a unique function $f : \SUPP (\mu_{c+g}) \to [0,1]$ such that
  for all $z\in \SUPP (\mu_{c+g})$,
  \[
  \dE\SBRA{\frac{1}{|G-z|^2 + f(z)^2}}=1.
  \]
  Moreover, $f$ is $C^\infty$ in the interior of $\SUPP (\mu_{c+g})$, and
  letting $\Phi(w,z):=(|w-z|^2 + f(z)^2)^{-2}$, the probability measure
  $\mu_{c+g}$ is absolutely continuous with density given by
  \begin{equation}\label{eq:formuladensity}
    z\mapsto
    \frac{1}{\pi} f(z)^2 \dE\SBRA{\Phi(G,z)} %
    + \frac{1}{\pi}
    \frac{\ABS{\dE\SBRA{(G-z)\Phi(G,z)}}^2}{\dE\SBRA{\Phi(G,z)}}.
  \end{equation}
\end{theorem}
It can be seen that $\mu$ is rotationally invariant when $K$ is a multiple of
the identity, while this is not the case if $\SUPP(\cL)\subset\dR$, in which
case $K_{22}=K_{12}=K_{21}=0$ (in this case $G$ does not have a density on
$\dC$ since $K$ is not invertible). Figure \ref{fi:simus} 
provides numerical simulations illustrating this phenomenon in two special
cases. Note also that the support of $\mu_{c+g}$ is not bounded since it
contains the support of $\cN(0,K)$. Thus, $\SUPP (\mu_{c+g})=\dC$ if $K$ is invertible. If $K$ is not invertible, it can be checked that the boundary
of $\SUPP (\mu_{c + g})$ is
\[
\BRA{z \in \dC : \dE\PAR{\frac{1}{|G-z|^2}} = 1}
\]
On this set, $f(z) = 0$, but from \eqref{eq:formuladensity}, we see that the
density does not vanish there. This phenomenon, not unusual for Brown
measures, occurs for the circular law and more generally for $R$-diagonal
operators, see Haagerup and Larsen \cite{MR1784419}.

The formula \eqref{eq:formuladensity} is slightly more explicit than the
formulas given in Biane and Lehner \cite[Section 5]{MR1876844}. Equation
\eqref{eq:formuladensity} will be obtained via a subordination formula for the
circular element (forthcoming proposition \ref{prop:subordination}) in the
spirit of the works of Biane \cite{MR1488333} or Voiculescu \cite{MR1744647}.
This subordination formula can also be used to compute more general Brown
measures of the form $\mu_{a + c}$ with $a, c$ $\star$-free, $c$ circular and $a$ normal.

\subsection{Extremal eigenvalues and the invariant measure}

Theorem \ref{th:eigenvals} suggests that the bulk of the spectrum of $L$ is
concentrated around the value $-mn$ in a two dimensional window of width $\si
\sqrt{n}$. Actually, it is possible to localize more precisely the support of
the spectrum, by controlling the extremal eigenvalues of $L$. Recall that $L$
has always the trivial eigenvalue $0$. Theorem \ref{th:support} below
describes the positions of the remaining eigenvalues. For simplicity, we
restrict our analysis to the Markovian case in which $\SUPP(\cL)\subset
[0,\infty)$ and to either model A or B. Analogous statements hold however in the general case.
Note that we have here $m>0$ and $K_{11}=1$ while $K_{22}=K_{12}=K_{21}=0$ (in
particular, $K$ is not invertible). We define for convenience the centered
random matrices
\begin{equation}\label{bars}
\Ul X = X-\dE X,
\quad 
\Ul D=D-\dE D, 
\quad
\Ul L=L-\dE L=\Ul X-\Ul D.
\end{equation}
If $J$ stands for the $n\times n$ matrix with all entries equal to $1$, then
we have
\[
\dE L = L-\Ul L= mJ - mnI.
\]

\begin{theorem}[Spectral support for model A]\label{th:support}
Assume model A, that $\SUPP(\cL)\subset\dR_+$ and that $\dE| \bx|^4 < \infty$. Then
  with probability one, for $n \gg1$, every eigenvalue $ \Ul \la$ of $\Ul
  L$ satisfies
  \begin{equation}\label{radius1}
    |\Re \Ul \la| \leq \si \sqrt{2 n\log(n)}\,(1+o(1)) %
    \quad\text{and}\quad %
    |\Im \Ul \la| \leq \si \sqrt{n}(2+o(1)) .
  \end{equation} 
  Moreover, with probability one, for $n \gg 1$, every eigenvalue $\la\neq
  0$ of $L$ satisfies
  \begin{equation}\label{radius2}
    |\Re  \la + m  n| \leq \si  \sqrt{2 n \log(n)}\,(1+o(1)) %
    \quad\text{and}\quad 
    |\Im \la| \leq \si \sqrt{n}(2+o(1)).
  \end{equation} 
\end{theorem}

To interpret the above result, recall that Yin and Bai \cite[theorem
2]{MR950344} prove that, in model A, if $\dE |\bx |^4 < \infty$ then the
operator norm of $\Ul X$ is $\si \sqrt{n}\,(2+o(1))$. On the other hand,
from the central limit theorem one expects that the operator norm and the
spectral radius of the diagonal (thus normal) matrix $\Ul D$ are of order
$\si \sqrt{2n\log(n)}\,(1+o(1))$ (as for maximum of i.i.d.\ Gaussian random
variables). Note that if one defines a \emph{spectral gap} $\kappa$ of the
Markov generator $L$ as the minimum of $|\Re \la|$ for $\la\neq 0$
in the spectrum of $L$, then by theorem \ref{th:support} one has a.s.\
\begin{equation}\label{sgap}
\kappa\geq mn - \si \sqrt{2n\log(n)}\,(1+o(1)).
\end{equation}

In theorem \ref{th:support}, we have restricted our attention to model A to be
in position to use Yin and Bai \cite{MR950344}. Beyond model A, their proof
cannot be extended to laws which do not satisfy the assumption $\dE(|\bx-m|^4)
= \cO( \si^4)$. The latter will typically not hold when $\si(n)$ goes to $0$.
For example, in model B, one has $\si^2(n) \sim p(n) \dE |\by|^2$ and
$\dE(|\bx-m|^4) \sim p(n)\dE |\by|^4$. In this situation, we have the
following result.

\begin{theorem}[Spectral support for model B]\label{th:support2}
Assume model B,  
with ${\bf y}$ a non-negative bounded variable, and that
  \begin{equation} \label{eq:hyp_sig4} \lim_{n \to \infty}
    \frac{n\si^2}{\log(n)} = +\infty;
  \end{equation}
  Then with probability one, for $n \gg1$, every eigenvalue $ \Ul \la$ of
  $\Ul L$ satisfies
  \begin{align}\label{radius12}
    & |\Re \Ul \la| %
    \leq (2+o(1))\si\sqrt{n\log(n)} %
    + \cO\PAR{\si^{\frac{1}{2}}n^{\frac{1}{4}}\log(n)} \\
    & \quad \quad\text{and}\quad\quad |\Im \Ul \la| %
    \leq (2 + o(1)) \si \sqrt{n} %
    + \cO\PAR{\si^{\frac{1}{2}} n^{\frac{1}{4}}\log(n)}\nonumber.
 \end{align} 
 Moreover, with probability one, for $n \gg 1$, every eigenvalue $\la\neq
 0$ of $L$ satisfies
 \begin{align}\label{radius22}
   &   | \Re \la + m  n|  %
   \leq(2+o(1))\si\sqrt{n\log(n)} %
   +\cO\PAR{\si^{\frac{1}{2}}n^{\frac{1}{4}}\log(n)} \\
   &\quad \quad\text{and}\quad\quad  %
   |\Im  \la|  %
   \leq (2+o(1))\si\sqrt{n} %
   +\cO\PAR{\si^{\frac{1}{2}}n^{\frac{1}{4}}\log(n)}\nonumber.
 \end{align} 
\end{theorem}
Note that $\si^{\frac{1}{2}} n^{\frac{1}{4}}\log(n)=o(\si\sqrt{n\log(n)})$
whenever \eqref{eq:hyp_sig4} is strenghtened to $n\si^2 \gg (\log(n))^2$. The term
$\si^{\frac{1}{2}}n^{\frac{1}{4}}\log(n)$ comes in our proof from an estimate
of Vu \cite{MR2384414} on the norm of sparse matrices with independent bounded
entries.

We turn to the properties of the invariant measure of $L$. If
$\SUPP(\cL)\subset\dR_+$ and $L$ is irreducible, then from the
Perron-Frobenius theorem, the kernel of $L$ has dimension $1$ and there is a
unique vector $\Pi\in(0,1)^n$ such that $L^\top\Pi = 0$ and $\sum_{i=1}^n\Pi_i
=1$. The vector $\Pi$ is the invariant measure of the Markov process with
infinitesimal generator $L$.

\begin{theorem}[Invariant measure]\label{th:invmeas} 
  Assume that either the assumptions of theorems \ref{th:support} or \ref{th:support2} hold.   Then, a.s.\ for $n \gg 1$, the Markov generator $L$ is irreducible and
  \[
  \TV{\Pi - U_n} %
  = \cO\PAR{\frac{\si}{m}\sqrt{\frac{\log(n)}{n}}} %
  + \cO\PAR{\frac{\sqrt{\si}}{m}\frac{\log(n)}{n^{3/4}}},
  \]
  where $U_n = \frac1n(1 ,\ldots,1)^\top$ is the uniform probability
  distribution on the finite set $\{1, \ldots , n\}$ and $\TV{Q} :=
  \frac{1}{2}\sum_{i=1}^n |Q_i|$ is the total variation norm.
\end{theorem}

\begin{figure}[htbp]
  \begin{center}
    \includegraphics[scale=0.7]{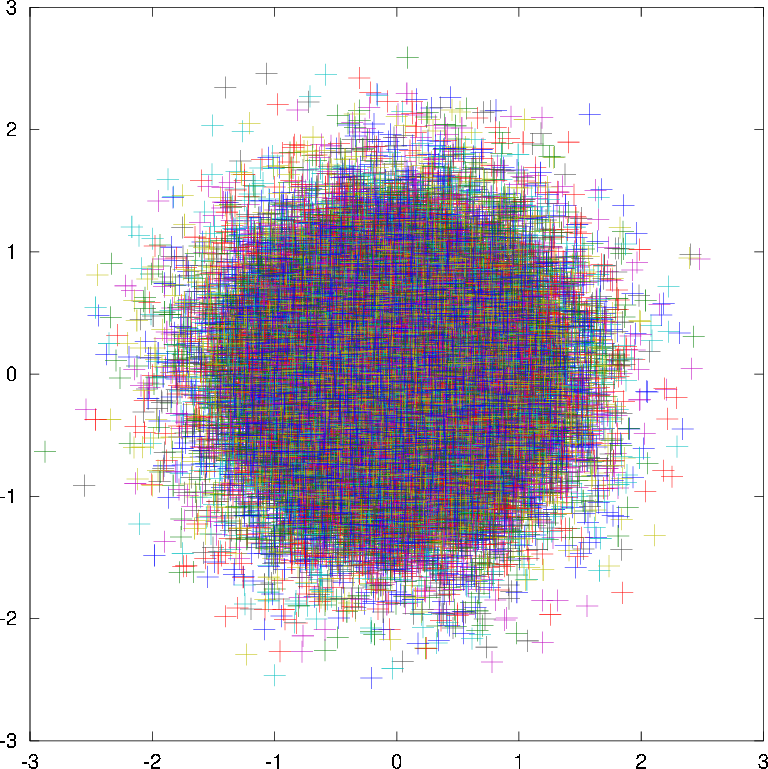}
    \vspace{2em}
    \includegraphics[scale=0.7]{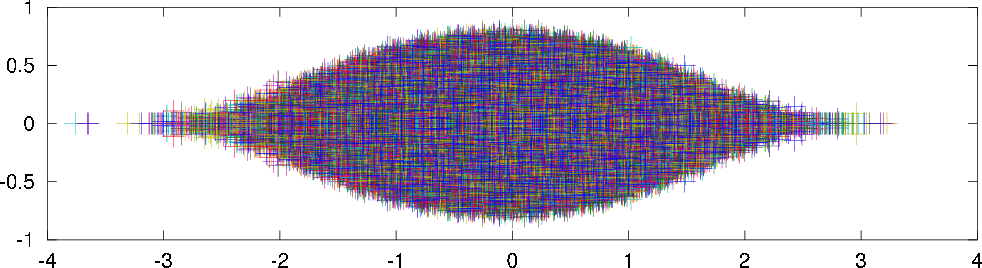}
    \caption{The bottom graphic shows a simulation of $50$ i.i.d.\ copies of
      the spectrum of $n^{-1/2}L$ with $n=500$ when $\cL$ is the exponential
      law on $\dR$ of parameter $1$ shifted by $-1$ ($m=0$, $K_{11}=1$, and
      $K_{12}=K_{21}=K_{22}=0$). The top graphics shows a simulation of $50$
      i.i.d.\ copies of the spectrum of $n^{-1/2}L$ with $n=500$ when $\cL$ is
      the Gaussian law on $\dC$ of mean $0$ and covariance $K=\frac{1}{2}I_2$.
      These simulations and graphics were produced with the free software GNU
      Octave provided by Debian GNU/Linux.}
    \label{fi:simus}
  \end{center}
\end{figure}


\subsection{Comments and remarks}

We conclude the introduction with a list of comments and open questions.

\subsubsection{Interpolation}

A first observation is that all our results for the matrix $L=X-D$ can be
extended with minor modifications to the case of the matrix $L^{(\de)}=X-\de
D$, where $\de\in\dR$ is independent of $n$, provided the law $\cN(0,K)$
characterizing our limiting spectral distributions is replaced by
$\cN(0,\de^2K)$. This gives back the circular law for $\de=0$.

\subsubsection{Almost sure convergence}
One expects that the convergence in theorem \ref{th:eigenvals} holds almost
surely and not simply in probability. This weaker convergence comes from our
poor control of the smallest singular value of the random matrix $M-z$. In the
special case when the law $\cL(n)$ of the entries has a bounded density
growing at most polynomially in $n$, then arguing as in \cite{MR2892961}, it
is possible to prove that the convergence in theorem \ref{th:eigenvals} holds
almost surely.

\subsubsection{Sparsity}
It is natural to conjecture that theorem \ref{th:eigenvals} continues to hold
even if \eqref{eq:hyp_sig2} is replaced by the weaker condition
\eqref{eq:hyp_sig}. However, this is a difficult open problem even in the
simpler case of the circular law which corresponds to analyze
$\mu_{n^{-1/2}\si^{-1}X}$. To our knowledge, our assumption
\eqref{eq:hyp_sig2} improves over previous works
\cite{gotze-tikhomirov-new,MR2977992}, where the variance is assumed to
satisfy $n^{1-\veps}\si^2 \to\infty$ for some $\veps >0$. Assumption
\eqref{eq:hyp_sig2} is crucial for the control of the smallest singular value
of $M$. We believe that with some extra effort the power $6$ could be reduced.
However, some power of $\log(n)$ is certainly needed for the arguments used
here. It is worthy of note that theorem \ref{th:singvals} holds under the
minimal assumption \eqref{eq:hyp_sig}.

\subsubsection{Dependent entries} One may ask if one can relax the i.i.d.\
assumptions on the entries of $X$. A possible tractable model might be based
on log-concavity, see for instance \cite{adamczak-chafai} and references
therein for the simpler case of the circular law concerning
$\mu_{n^{-1/2}\si^{-1}X}$. Actually, one may expect that the results remain
valid if one has some sort of uniform tightness on the entries. To our
knownledge, this is not even known for $\mu_{n^{-1/2}\si^{-1}X}$, due to the
difficulty of the control of the small singular values.

\subsubsection{Heavy tails}
A different model for random Markov generators is obtained when the law $\cL$
of $\bx$ has heavy tails, with e.g.\ infinite first moment. In this context,
we refer e.g.\ to
\cite{bordenave-caputo-chafai-heavygirko,bordenave-chafai-changchun} for the
spectral analysis of non-Hermitian matrices with i.i.d.\ entries, and to
\cite{bordenave-caputo-chafai-ii} for the case of reversible Markov transition
matrices. It is natural to expect that, in contrast with the cases considered
here, there is no asymptotic independence of the matrices $X$ and $D$ in the
heavy tailed case.

\subsubsection{Spectral edge and spectral gap}
Concerning theorem \ref{th:support}, it seems natural to conjecture the
asymptotic behavior $\kappa= mn - \si\sqrt{2n\log(n)}\,(1+o(1))$ for the
spectral gap \eqref{sgap}, but we do not have a proof of the corresponding upper bound. 
In the same spirit, in the setting of theorem \ref{th:support} or
theorem \ref{th:support2} we believe that with probability one, with $\Ul
L=L-\dE L$,
\[
\lim_{n\to\infty} \frac{s_1(\Ul L)}{\si\sqrt{2 n\log(n)}} %
=\lim_{n\to\infty}\frac{|\la_1(\Ul L)|}{\si\sqrt{2 n\log(n)}}=1,
\]
which contrasts with the behavior of $\Ul X$ for which $s_1/|\la_1|\to2$ as
$n\to\infty$ under a finite fourth moment assumption \cite{MR950344}.

\bigskip

The rest of the article is structured as follows. Sections
\ref{se:proof:th:singvals} and \ref{se:proof:th:eigenvals} provide the proof
of theorem \ref{th:singvals} and of theorem \ref{th:eigenvals} respectively.
Section \ref{se:proof:th:naturemu+th:propmu} is devoted to the proof of
theorem \ref{th:naturemu} and of theorem \ref{th:propmu}. Section
\ref{se:proof:th:support} gives the proof of theorems \ref{th:support} and
\ref{th:support2}, section \ref{sec:invmeas} contains a proof theorem
\ref{th:invmeas}. Finally, an appendix collects some facts on concentration
function and small probabilities.

\section{Convergence of singular values: Proof of theorem \ref{th:singvals}}
\label{se:proof:th:singvals}

We 
adapt the strategy of proof described in \cite[section
4.5]{bordenave-chafai-changchun}. 

\subsection{Concentration of singular values measure}
A standard separability and density argument shows that it is sufficient to
prove that for any compactly supported $\cC^\infty(\dR)$ function $f$, a.s.\
\[
\int\!f\,d\nu_{M-z}-\int\!f\,d\nu_{z}\to 0.
\]
Since the matrix $M - z$ has independent rows, we can rely on the
concentration of measure phenomenon for matrices with independent rows, see
\cite{bordenave-caputo-chafai-heavygirko}, \cite{bordenave-chafai-changchun},
or \cite{MR2535081}. In particular, using \cite[lemma
4.18]{bordenave-chafai-changchun} and the Borel-Cantelli lemma, we obtain that
for any compactly supported continuous function $f$, a.s.\
\[
\int\!f\,d\nu_{M-z} - \dE\int\!f\,d\nu_{M-z} \to 0.
\]
In other words, in order to prove the first part of  theorem \ref{th:singvals}, it is sufficient to prove the convergence to $\nu_z$
of the averaged measure $\dE \nu_{M -z}$.

\subsection{Centralization and truncation}

We now prove that it is sufficient to prove the convergence for centered
entries with bounded support. We first notice that
\[
\dE M =  \si^{-1} n^{-1/2}  m J
\] 
has rank one, where $J$ stands for the matrix with all entries equal to $1$.
Hence, writing $\Ul M = M - \dE M$, from standard perturbation inequalities (see e.g.\ \cite[Th.\
3.3.16]{MR1288752}):
\begin{equation}
  \label{eq:IPPinterlacing} 
  \ABS{ \int\!f\,d\nu_{\Ul M-zI} - \int\!f\,d\nu_{M -zI}}%
  \leq \BV{f} \frac{\RANK(M- \Ul M)}{n} %
  \leq \frac{\BV{f}}{n},
\end{equation}
where $\BV{f} = \int\!|f'(t)|\,dt $ denotes the bounded variation norm of
$f$. In particular, it is sufficient to prove the convergence of $\dE\nu_{\Ul
  M-zI}$ to $\nu_z$. Recall the definition \eqref{bars} of the centered matrices $\Ul X, \Ul D$. Define
\[
X'_{ij} = \Ul X_{ij} \IND_{\BRA{|\Ul X_{ij} |\leq \veps \si  \sqrt{n}}},
\]
where $\veps = \veps_n$ is a sequence going to $0$ such that 
\begin{equation}\label{eq:hyp_lind2}
  \lim_{n \to \infty}  %
  \dE\SBRA{\frac{|X_{ij}-m|^2}{\si^2}\IND_{\{|X_{ij}-m|^2\geq\veps^2\si^2 n}\}} = 0.
\end{equation}
(its existence is guaranteed by assumption \eqref{eq:hyp_lind}). Then, let $X'
= (X'_{ij} )_{ 1 \leq i , j \leq n}$, $L' = X' - \Ul D$ and  $M' = L'/(\si\sqrt{n})$. From
Hoffman-Wielandt inequality, we have
\begin{align*}
  \frac{1}{n}\sum_{k=1}^n\ABS{s_k(M'-zI)-s_k(\Ul M-zI)}^2 %
  & \leq  \frac{1}{n} \sum_{1 \leq i,j \leq n}| \Ul M_{ij} - M'_{ij} |^2 \\
  & =   \frac{1}{n^2 \si^2}\sum_{1 \leq i , j \leq n}
  |\Ul X_{ij}|^2\IND_{\BRA{|\Ul X_{ij}|>\veps \si  \sqrt{n}}} 
\end{align*}
Then by \eqref{eq:hyp_lind2}, we deduce that 
\[
\dE\,\frac{1}{n}\sum_{k=1}^n\ABS{s_k(M'-zI)-s_k( \Ul M-zI)}^2 \to 0.
\]
The left hand side above is the square of the expected Wasserstein $W_2$
coupling distance between $ \nu_{M' - z I }$ and $ \nu_{ \Ul M - z I }$. Since
the convergence in $W_2$ distance implies weak convergence, we deduce that it
is sufficient to prove the convergence of $\dE \nu_{M' - z I }$ to $\nu_z$. We
then center the entries of $X'$, and set $\Wt L = L' - m' J$ and $\Wt M = \Wt
L/(\si\sqrt{n})$, where $m' = \dE X'_{11} $. As in \eqref{eq:IPPinterlacing},
we find
\[
  \ABS{ \int\!f\,d\nu_{\Wt M-zI} - \int\!f\,d\nu_{M' -zI}}%
  \leq \BV{f} \frac{\RANK(\Wt M- M')}{n} %
  \leq \frac{\BV{f}}{n}.
\]
Finally, consider the generator associated to $X'':=X'-\dE X'$. Namely, define
the matrices $L'' = X'' - D''$ and $M'' = L''/(\si\sqrt{n})$, where $D''$ is
the diagonal matrix with
\[
D''_{ii} = \sum_{j } (X'')_{ij} %
= \Ul D_{ii} -\sum_{j}\Ul X_{ij}\IND_{\BRA{|\Ul X_{ij} | >\veps\si\sqrt{n}}}-nm'.
\]
Using $m ' = - \dE \Ul X_{11}\IND_{\BRA{|\Ul X_{ij} | > \veps \si \sqrt{n}}}$
and again the Hoffman-Wielandt inequality, one has
\begin{align*}
  & \frac{1}{n}\sum_{k=1}^n\ABS{s_k(\Wt M-zI)-s_k(M''-zI)}^2 %
  \leq \frac{1}{n^2 \si^2}\sum_{1 \leq i \leq n}\ABS{D''_{ii}-\Ul D_{ii}}^2\\
  & \quad \quad \leq \frac{1}{n^2 \si^2}\sum_{1 \leq i \leq n} %
  \ABS{\sum_{1\leq j\leq n} \PAR{ \Ul X_{ij}\IND_{\{|\Ul X_{ij}|>\veps
        \si\sqrt{n}\}} - \dE \Ul X_{ij}\IND_{\{|\Ul X_{ij}|>\veps
        \si\sqrt{n}\}} } }^2.
\end{align*}
The expectation of the above expression equals 
\[
\frac{1}{n^2 \si^2}\sum_{1 \leq i,j \leq n} %
  \dE\,\ABS{\Ul X_{ij}\IND_{\{|\Ul X_{ij}|>\veps \si\sqrt{n}\}}  - \dE \Ul X_{ij}\IND_{\{|\Ul X_{ij}|>\veps \si\sqrt{n}\}} } ^2,
\]
which tends to $0$ by \eqref{eq:hyp_lind2}.

In summary, for the remainder of the proof, we will assume without loss of generality that the law of
$X_{ij}$ satisfies
\begin{equation}\label{eq:propX11}
  \dE X_{11} = 0  \, , %
  \quad \dP ( |X_{11} |\geq \kappa(n) )  = 0 %
  \quad \text{and} \quad \dE |X_{11} |^2   = {\si'(n)}^2,
\end{equation}
where 
\begin{equation*}
  \label{eq:hyp_ks}
  \kappa(n) = o (\si(n)\sqrt{n}) %
  \quad \text{and} \quad  %
  \si'(n) = \si(n)( 1 + o(1)). 
\end{equation*}

\subsection{Tightness}
\label{subsec:tightness}
Let us check that $\dE \nu_{M - z I}$ is tight. Recall an instance of the Weyl
inequality: for all $A,B$ in $\cM_n (\dC)$, for all $i$:
\[
| s_i ( A ) - s_i (B) | \leq s_{1} ( A - B).
\]
Consequently,
\[
\int\!s^2\,d\nu_{M - z I}(s) \leq \int\!s^2\,d\nu_{M}(s) + |z|^2.
\]
It is thus sufficient to check that $\dE \int\!s^2\,d\nu_{M}(s)$ is uniformly
bounded. However,
\begin{align*}
  \int\!s^2\,d\nu_{M}(s)
  & =  \frac{1}{n}\sum_{1 \leq i,j \leq n}|M_{ij}|^2 \\
  & = \frac{1}{n^2\si^2}\sum_{1 \leq i \ne j \leq n} |X_{ij}|^2 %
  + \frac{1}{n^2\si^2}\sum_{1 \leq i \leq n} %
  \ABS{\sum_{1 \leq j \leq n , j \ne i } X_{ij}}^2 .
\end{align*}
The conclusion follows by taking expectation and using \eqref{eq:propX11}.

\subsection{Linearization}\label{linearization}

We use a common linearization technique. With the notation from
\eqref{eq:agnuz} and \eqref{eq:bgnuz} one has the identity of the
Cauchy-Stieltjes transform, for $\eta \in \dC_+$,
\begin{equation}\label{eq:resstieljes}
  S_{\check\nu_{M-zI}}(\eta) = \frac{1}{2n}\TR \left[(H(z)-\eta I_{2n})^{-1}\right], 
\end{equation}
where $I_{2n}$ is the $2n\times 2n$ identity matrix and $H(z)$ is the $2n\times 2n$ hermitian matrix
\[
H(z) :=
\begin{pmatrix} 
  0 & M - z \\ 
  ( M- z)^* &  0  
\end{pmatrix},
\]
with eigenvalues $\{\pm s_i(M-z),\,i=1,\dots,n\}$.
Define $\dH_+ \subset\cM_2(\dC)$ as 
\begin{equation}\label{H_+}
\dH_+ := \BRA{
\begin{pmatrix}
  \eta & z \\ \bar z & \eta
\end{pmatrix}, z \in \dC, \eta \in \dC_+}.
\end{equation}
For $q\in\dH_+$, with 
\[
q(z,\eta) := 
\begin{pmatrix} \eta & z \\ \bar z  & \eta \end{pmatrix},
\] 
let $q (z,\eta) \otimes I_n$ denote the   $2n\times 2n$ matrix obtained by repeating $n$ times along the diagonal the $2\times 2$ block $q$.
Through a permutation of the entries, the matrix $H(z)-\eta I_{2n}$ is equivalent to the
matrix 
\begin{equation}\label{Bq}
B - q (z,\eta) \otimes I_n,
\end{equation}
where $B$ is obtained from the $2\times 2$ blocks $B_{ij}$, $1\leq i,j\leq n$:
\[
B_{ij} :=   
\begin{pmatrix} 0 & M_{ij} \\ \bar M_{ji}  & 0 \end{pmatrix}.
\]
If $B(z):=B - q (z,0) \otimes I_n$, then 
$B (z) \in \cM_n ( \cM_{2} (\dC)) \simeq \cM_{2n} ( \dC)$ is
Hermitian, and its resolvent is denoted by
\begin{equation}\label{R(q)}
R(q) = (B(z) -  \eta I_{2n} )^{-1}  = (B - q (z,\eta)  \otimes I_n )^{-1} .
\end{equation}
Then $R(q)\in\cM_n(\cM_{2}(\dC))$ and, by \eqref{eq:resstieljes}, we deduce
that
\[
S_{\check \nu_{M-zI}} (\eta) = \frac{1}{2n}\TR R (q).
\]
We set 
\[
R(q)_{kk} = 
\begin{pmatrix} 
  a_{k} (q) & b _{k} (q) \\ 
  c _{k} (q)& d _{k} (q)
\end{pmatrix}
\in \cM_{2} (\dC).
\]
As in \cite[lemma
4.19]{bordenave-chafai-changchun}, it is not hard to check that
\begin{equation}\label{eq:a=d}
  a(q) := \frac{1}{n}\sum_{k= 1}^n a _{k} (q) %
  = \frac{1}{n}\sum_{k= 1}^nd_{k} (q)%
  \quad \text{and  } \quad %
  b (q) := \frac{1}{n}\sum_{k= 1}^nb _{k} (q) %
  = \frac{1}{n} \sum_{k= 1}^n  \bar c_{k}(q),
\end{equation} 
It follows that
\begin{equation} \label{eq:resstieljes2}
  S_{\check \nu_{M  - zI}} (\eta)  =  a (q).
\end{equation}
Hence, in order to prove that $\dE \nu_{M- z}$ converges, it is sufficient to
prove that $\dE a(q)$ converges to, say, $\al(q)$, for all $q\in\dH_+$.  By tightness, $\al(q)$ will 
necessarily be the Cauchy-Stieltjes transform of a symmetric measure. (Indeed,
since $ \check \nu_{M- zI}$ is tight and symmetric, any accumulation point of
$ \check \nu_{M- zI}$ will be a symmetric probability measure. Also, recall
that the weak convergence of a sequence of probability measures
on $\dR$, $(\nu_n)_{n \geq 1}$ to $\nu$, is equivalent to the convergence for all $\eta \in
\dC_+$ of $S_{\nu_n } (\eta)$ to $S_{\nu } (\eta)$).

\subsection{Approximate fixed point equation}

We use a resolvent method to deduce an approximate fixed point equation
satisfied by $a(q)$. The Schur block inversion formula states that if $A$ is a
$k\times k$ matrix then for every partition $\{1,\ldots,k\}=I\cup I^c$,
\begin{equation*}\label{eq:schur}
  (A^{-1})_{I,I}=(A_{I,I}-A_{I,I^c}(A_{I^c,I^c})^{-1}A_{I^c,I})^{-1}.
\end{equation*}
Applied to $k=2n$, $A=B(z)-\eta I_{2n}$, $q=q(z,\eta)$, it gives
\begin{equation}\label{resolnn}
R(q)_{nn} = \PAR{ 
  \begin{pmatrix} 
    0 & M_{nn} \\ 
    \bar M_{nn} & 0 
  \end{pmatrix} - q - Q^* \Wt R (q)Q } ^{-1},
\end{equation}
where $Q \in  \cM_{n-1, 1} ( \cM_{2} (\dC))$, is the $2(n-1)\times 2$ matrix given by the blocks
\[
Q_ i = 
\begin{pmatrix} 
  0 & M_{n i } \\ 
  \bar M_{i n} & 0 
\end{pmatrix} 
= \frac{1}{\si\sqrt{n}} 
\begin{pmatrix} 
  0 & X_{n i } \\ 
  \bar X_{i n} & 0 
\end{pmatrix}
\]
for $i=1,\dots,n-1$, and
 $\Wt B = (B_{ij})_{1 \leq i,j \leq n-1}$, $\Wt B (z) =
\Wt B - q (z, 0) \otimes I_{n-1}$,
\[
\Wt R (q)%
= ( \Wt B - q \otimes I_{n-1}) ^{-1} %
= ( \Wt B (z) - \eta I_{2(n-1)}) ^{-1}
\]
is the resolvent of a minor. 

Define the matrix $M' \in \cM_{n-1} (\dC)$ by
\[
M'_{ij} = \frac{ X_{ij} }{\si \sqrt{n}} -\delta_{i,j}\sum_{1 \leq k \leq n-1} \frac{ X_{ik} }{\si \sqrt{n}} ,
\]
for $i,j=1,\dots,n-1$.
Let $R'$ and $B' $ in $ \cM_{n-1} ( \cM_{2} (\dC)) $ be the matrices
obtained as in \eqref{Bq} and \eqref{R(q)} with $M$ replaced by $M' $.
From the resolvent formula and the bounds $\|R '\|,\|\Wt R\|\leq (\Im (\eta))^{-1}$:
\[
\NRM{\Wt R - R' } = \NRM{ R' ( \Wt B - B') \Wt R} %
\leq \frac{1}{\si \sqrt{n} \Im (\eta) ^2 } %
\NRM{\DIAG\PAR{\begin{pmatrix} 0 & X_{n i } \\ \bar X_{i n} &
      0 \end{pmatrix}}_{1 \leq i \leq n-1}}.
\]
Hence using \eqref{eq:propX11}, we deduce the uniform estimate 
\begin{equation}\label{eq:wtRtoR}
\|\Wt R - R'  \| \leq   \frac{\kappa}{\si \sqrt{n} \Im (\eta) ^2 }   = o(1). 
\end{equation}
Here and below, $o(1)$ denotes a vanishing deterministic sequence, that depends on $q(z,\eta)$ through $\Im (\eta)$ only. Since
$\| Q^* S Q \| \leq \|S \| \|Q^* Q \|$ and $\dE \| Q^* Q \| = \cO(1)$, with
$S=\Wt R - R'$ we obtain
\[
R_{nn} = - \PAR{ -  \begin{pmatrix} 
    0 & M_{nn} \\ 
    \bar M_{nn} & 0 
  \end{pmatrix} + q + {Q}^* R'
  Q + \veps_1} ^{-1}.
\]
with $\veps_1$ a $2\times 2$ matrix satisfying $\dE \|\veps_1 \|= o(1)$. 

We denote by $\cF_{n-1}$ the 
$\si$-algebra spanned by the
variables $(X_{ij})_{1 \leq i ,j \leq n-1}$. Then $R'$ is
$\cF_{n-1}$-measurable and is independent of $Q$. If $\dE_n [\,\cdot\,] := \dE
[\,\cdot\,|\cF_{n-1}]$, we get, using \eqref{eq:propX11} and \eqref{eq:a=d}
\begin{align*}
  \dE_n \SBRA{Q^* R' Q }
  & = \sum_{1 \leq k,\ell\leq n-1}\dE_n\SBRA{Q_k^*R'_{k \ell}Q_\ell}
   = \frac{{\si'}^2}{ \si^2 n} \sum_{k = 1}^{n-1} %
  \begin{pmatrix} a'_{k} & 0\\ 0 & d'_{k} \end{pmatrix}\\
  & = \frac{{\si'}^2}{ \si^2 n} \sum_{k = 1}^{n-1} %
  \begin{pmatrix} a'_{k} & 0\\ 0 & a'_{k} \end{pmatrix} = \frac{1}{ n} \sum_{k = 1}^{n-1} %
  \begin{pmatrix} \Wt a_{k} & 0\\ 0 & \Wt a_{k} \end{pmatrix} + \veps_2,
\end{align*}
where 
\[
R'_{kk} = 
\begin{pmatrix} 
  a'_{k} & b'_{k} \\ 
  c'_{k} & d'_{k} %
\end{pmatrix}
\quad \text{and} \quad %
\Wt R_{kk} =
\begin{pmatrix} 
  \Wt a_{k} & \Wt b_{k} \\ 
  \Wt c_{k} & \Wt d_{k} %
\end{pmatrix},
\]
and $\veps_2$ is a $2\times 2$ matrix.
Using \eqref{eq:wtRtoR}, we have the bound 
\[
\ABS{\frac{1}{n}\sum_{k=1}^{n-1}\Wt a_{k}-\frac{1}{n}\sum_{k=1}^{n-1}a'_{k}} %
= \ABS{\frac{1}{2n}\TR(\Wt R)-\frac{1}{2n}\TR(R')} %
\leq \NRM{\Wt R - R'} = o(1)
\]
We deduce that $\|\veps_2 \|= o(1)$. Similarly, recall that $\Wt B(z)$ is a
minor of $B(z)$. We may thus use the interlacing inequality
\eqref{eq:IPPinterlacing} for the function $f = (\cdot - \eta)^{-1}$. We find
\[
\ABS{\sum_{k = 1}^{n-1} \Wt a_{k} - \sum_{k = 1}^{n} a_{k}} %
\leq 2 \int_{\dR} \frac{1}{|x - \eta |^2} dx %
= \cO\PAR{\frac{1}{\Im(\eta)}}.
\]
In summary, we have checked that
\[
\dE_n \SBRA{Q^* \Wt R  Q }  %
=   \begin{pmatrix}   a  &  0    \\  0  &   a  \end{pmatrix} + \veps_3, 
\]
where $a=a(q)$ is as in \eqref{eq:resstieljes2}, and $\veps_3$ satisfies
$\NRM{\veps_3}=o(1)$. Moreover, we define
\[
\veps_4:=
\dE_n \SBRA{\PAR{Q^* \Wt R  Q-  \dE_n \SBRA{Q^* \Wt R  Q}}^*  %
  \PAR{Q^* \Wt R Q- \dE_n \SBRA{Q^* \Wt R Q}}}. %
\]
Since $\|\Wt R\| \leq \Im (\eta) ^{-1}$, we have
\[
\| \Wt R^*_{ii} \Wt R_{ii} \| \leq \Im ( \eta)^{-2} %
\quad \text{and} \quad %
\TR \Big( \sum_{i,j} \Wt R^*_{ij} \Wt R_{ji} \Big) = \TR ( \Wt R^* \Wt R ) \leq 2 n
\Im ( \eta)^{-2}.
\] 
Also, by \eqref{eq:propX11} 
\[
\dE  | X^2_{ij} - {\si'}^2 |^2 \leq 2 \kappa^2 {\si'}^2.   
\]
Then, an elementary
computation gives
\[
\NRM{\veps_4} %
\leq \TR (  \veps_4 )  %
= \cO\PAR{ \frac{\kappa^2{\si'}^2}{n\Im(\eta)^2 \si^4}} = o(1).
\]
Moreover, $a(q)$ is close to its expectation. More precisely, from
\cite[lemma 4.21]{bordenave-chafai-changchun}, 
 \[
\dE|a(q)-\dE a(q)|^2 = \cO\PAR{ \frac{1}{n\Im(\eta)^2}} = o(1).
\]
We recall finally that the central limit theorem with Lindeberg condition
implies that
\[
M_{nn}  = - \frac{1}{\si\sqrt{n}}\sum_{i=1} ^ {n-1} X_{ni}
\] 
converges weakly to $G$ with distribution $\cN(0,K)$. From Skorokhod's
representation theorem, we may assume that this convergence holds almost
surely. Then, we have proved that the $2\times 2$ matrix
\[
A = - 
\begin{pmatrix} 0 & M_{nn} \\ \bar M_{nn} & 0 \end{pmatrix} %
+ 
\begin{pmatrix} 0 & G\\ \bar G & 0 \end{pmatrix}  + Q^* \Wt R Q %
- \dE \begin{pmatrix} a & 0 \\ 0 & a \end{pmatrix}
\]
has a norm which converges to $0$ in probability as $n\to\infty$. 
On the other hand, from \eqref{resolnn}, 
\begin{equation}\label{resolnn2}
R_{nn}A= -R_{nn}\PAR{ q+\dE\begin{pmatrix} a & 0 \\ 0 & a \end{pmatrix}
  - \begin{pmatrix} 0 & G\\ \bar G & 0 \end{pmatrix} } - I_2
\end{equation}
Since the norms of $\PAR{ q + \dE \begin{pmatrix} a & 0 \\ 0 &
    a \end{pmatrix} - \begin{pmatrix} 0 & G\\ \bar G & 0 \end{pmatrix} }
^{-1}$ and $R_{nn}$ are at most $\Im (\eta)^{-1}$, we get
\[
\dE R_{nn} %
= -\dE \PAR{ q+\dE\begin{pmatrix} a & 0 \\ 0 & a \end{pmatrix}
  - \begin{pmatrix} 0 & G\\ \bar G & 0 \end{pmatrix} }^{-1}+\veps
\]
with $\NRM{\veps}= o(1)$. Using exchangeability, we get that the functions in \eqref{eq:a=d} satisfy
\[
\dE \begin{pmatrix} a & b \\ \bar b & a \end{pmatrix} %
= -\dE \PAR{ q + \dE \begin{pmatrix} a & 0 \\ 0 & a \end{pmatrix}
  - \begin{pmatrix} 0 & G\\ \bar G & 0 \end{pmatrix} }^{-1} + \veps.
\]

\subsection{Uniqueness of the fixed point equation}
\label{subsec:unicity}
From what precedes, any accumulation point of $\dE \begin{pmatrix} a & b \\
  \bar b & a \end{pmatrix}$ is solution of the fixed point equation
\begin{equation}\label{eq:FPCircular}
  \begin{pmatrix} \al & \beta \\ \bar \beta & \al \end{pmatrix}   %
  = -\dE  \PAR{ q + \begin{pmatrix} \al & 0 \\ 0 & \al \end{pmatrix}   - \begin{pmatrix} 0 & G\\ \bar G & 0 \end{pmatrix}   }^{-1}.
\end{equation}
with $\al = \al (q) \in \dC_+$. Therefore, for $q=q(z,\eta)\in\dH_+$:
\[
\al = \dE \frac{ \al+ \eta}{| G - z|^2 - ( \al + \eta)^2 }.
\]
The above identity is precisely the fixed point equation satisfied
by $S_{\check \nu_z} (\eta)$ given in theorem \ref{th:singvals}. Hence, to
conclude the proof of theorem \ref{th:singvals}, it is sufficient to prove
that there is a unique symmetric measure whose Cauchy-Stieltjes transform is
solution of this fixed point equation. We know from \eqref{eq:resstieljes2}
and Montel's theorem that $\eta\in\dC_+\mapsto \al (q(z, \eta))\in\dC_+$ is
analytic for every fixed $z\in\dC$. In particular, it is sufficient to check that there is a unique
solution in $\dC_+$ for $\eta = it$, for a fixed $t >0$. 
If $h(z,t) = \Im (\al (q ))$, we find
\[
h = \dE\, \frac{ h + t }{| G- z|^2 + ( h  + t)^2 }.
\]
Thus, $h  \ne 0$ and 
\[
1 = \dE \,\frac{ 1 + t  h ^{-1} }{| G - z|^2 + ( h  + t)^2 }.
\]
The right hand side in a decreasing function in $h$ on $(0, \infty)$ with
limits equal to $+ \infty$ and $0$ at $h \to 0$ and $h \to \infty$. Thus,
there is a unique solution $h>0$ of the above equation. The proof of theorem
\ref{th:singvals} is over.

\section{Convergence of eigenvalues: Proof of theorem \ref{th:eigenvals}}
\label{se:proof:th:eigenvals}

\subsection{Strategy of proof}

In order to prove theorem \ref{th:eigenvals}, we will use the Hermitization
method; see e.g.\  \cite[lemma 4.3]{bordenave-chafai-changchun} for
the proof of the next lemma.

\begin{lemma}[Hermitization]\label{le:girko2}
  Let ${(A_n)}_{n\geq1}$ be a sequence of complex random matrices where $A_n$
  is $n \times n$ for every $n\geq1$. Suppose that there exists a family
  ${(\nu_z)}_{z\in\dC}$ of (non-random) probability measures on $\dR_+$ such
  that for a.a. $z \in \dC$,
  \begin{itemize}
  \item[$(i)$] $\nu_{A_n-z}$ tends weakly in probability to $\nu_z$;
  \item[$(ii)$] $\log(\cdot)$ is uniformly integrable in probability for
    $(\nu_{A_n-z})_{n\geq1}$.
  \end{itemize}
  Then, in probability, $\mu_{A_n}$ converges weakly to the probability measure $\mu$ defined by
  \[
  \mu= \frac{1}{2\pi}\Delta \int_0^\infty\!\log(t)\,d\nu_z(t).
  \]
\end{lemma}

Applied to our matrix $M (n)= \si^{-1} \sqrt{n}(L/ n+mI)$, the validity of
$(i)$ follows from theorem \ref{th:singvals} (convergence a.s.\ implies
convergence in probability). The proof of $(ii)$ is performed in the remainder
of this section using ideas developed by Tao and Vu in
\cite{tao-vu-cirlaw-bis} and by the authors in \cite{MR2892961}, together with
an analysis of the smallest singular value which follows closely the work of
G\"otze and Tikhomirov \cite{gotze-tikhomirov-new}. We are going to prove the
following theorem. 

\begin{theorem}[Uniform integrability]\label{th:unifint}
  Under the assumptions of theorem \ref{th:eigenvals}, there exists an
  increasing function $J:[0,\infty)\to[0,\infty)$ with $J(t)/t\to\infty$ as
  $t\to\infty$ such that for all $z\in\dC\backslash \La$,
  \[
  \lim_{t \to \infty} \limsup_{n \to \infty} %
  \dP\PAR{\int_0^\infty J(|\log(s)|)\,d\nu_{M(n)-z}(s) > t} = 0.
  \]
\end{theorem}

From the de La Vall\'ee Poussin criterion for uniform integrability, theorem
\ref{th:unifint} implies point $(ii)$ above, i.e.\ the uniform integrability
in probability of $\log(\cdot)$ for $(\nu_{M(n)-z})_{n\geq1}$, see
\cite{bordenave-chafai-changchun}. Therefore, it implies theorem
\ref{th:eigenvals}. The proof of theorem \ref{th:unifint} is divided into
three steps corresponding to the control of the large singular values, of the
moderately small singular values and of the smallest singular value of $M(n) -
z$.

\subsection{Large singular values}

\begin{lemma}[Tightness]\label{le:large} %
  For all $z \in \dC$, there exists a constant $C >0 $ uniform on bounded sets
  in $z$, such that for any $n \geq 1$ and any $t > 0$,
  \[
  \dP \PAR{\int\!x^{2}\,d\nu_{M(n)-z}(x) > t } \leq C\, t^{-1}.
  \]
  In particular, for any $u>0$,
  \begin{equation}\label{s1bound}
  \dP ( \| M(n) - z \| \geq u  ) \leq C \,u^{-2}n.
  \end{equation}
\end{lemma}

\begin{proof} As in Section \ref{subsec:tightness}, we find
  \begin{align*}
    \int\!x^{2}\,d\nu_{M - z}(x)
    &  \leq  \int\!x^{2}\,d\nu_{M}(x)  + |z|^2 \\
    & = \frac{1}{\si^2 n^2} \sum_{i \ne j} |X_{ij} |^2 + \frac{1}{\si^2
      n^2} \sum_{i} \Big|\sum_{j \ne i } ( X_{ij} - m) \Big|^2 + |z|^2.
  \end{align*}  
  Since $\dE |X_{ij} |^2=\cO(\si^2)$ by \eqref{eq:hyp_tightness1}, taking
  expectation, we find
  \[
  \dE \int\!x^{2}\,d\nu_{M - z}(x)  = \cO(1).
  \]
  It remains to apply Markov's inequality. The bound \eqref{s1bound} now follows
  from $ \| M(n) - z \|^2\leq n\int\!x^{2}\,d\nu_{M(n)-z}(x)$.
\end{proof}

\subsection{Smallest singular value}
A crucial step towards the proof of theorem \ref{th:unifint} is a lower bound,
in probability, on the smallest singular value $s_{n} ( M - z )$. Here, our
main result is a quasi polynomial lower bound.
  
\begin{proposition}[Smallest singular value]\label{prop:small}
  Under the assumptions of theorem \ref{th:eigenvals}, for all $z \in \dC
  \backslash \La$,
  \begin{equation}\label{quasipol}
    \lim_{n\to\infty}\dP(s_{n} ( M  - z   ) \geq e^{-(\log(n))^2})=1.
  \end{equation}
\end{proposition}
The proof of Proposition \ref{prop:small} follows closely the strategy
developed in \cite{gotze-tikhomirov-new}, in turn inspired by the works
\cite{MR2407948,MR2146352}. However, some crucial modifications are needed due
to the extra terms coming form the diagonal matrix $D$. If one assumes that
$\si^2(n)\geq n^{-1/4+\veps}$ for some $\veps>0$, then the quasi polynomial
$e^{-(\log(n))^2}$ can be replaced by $n^{-4}$ in \eqref{quasipol}, see Section
\ref{proof:small} below. Before we start the proof of Proposition
\ref{prop:small}, we collect some important preliminaries.

\subsubsection{Distance of a random vector to a subspace}
The following lemma is valid under the sole assumption
\eqref{eq:hyp_tightness1}. For the proof we adapt the argument of
\cite[Proposition 5.1]{tao-vu-cirlaw-bis}, see also \cite[Appendix
A]{bordenave-chafai-changchun}, but some extra care is needed to handle the
sparse regime $\si^2\to 0$. Remark that condition \eqref{eq:hyp_tightness1}
implies that there exists $a >0$ such that
\begin{equation} \label{eq:hyp_tightness} %
  \inf_{T\geq a}\inf_{n \geq 1}\si^{-2}\VAR(\bx^{(T)} ) >0,
\end{equation}
where $\bx^{(T)}$ is the random variable $\bx$ conditioned on $|\bx|\leq T$. 
Indeed, let $\epsilon(n,T)\geq 0$ denote a generic sequence such that 
$\lim_{T\to\infty}\epsilon(n,T)=0$ uniformly in $n$. Then, 
from the first display in \eqref{eq:hyp_tightness1}, one has $\dE[|\bx^{(T)}|^2]\geq (1-\epsilon(n,T))\dE[|\bx|^2]$. Schwarz' inequality and the second display in 
\eqref{eq:hyp_tightness1} also imply that 
$\dE[|\bx|1_{|\bx|>T}]\leq \epsilon(n,T)\si$. In conclusion, $\dE[|\bx^{(T)}|^2] - |\dE[\bx^{(T)}]|^2\geq (1-\epsilon(n,T))\si^2$ for some $\epsilon(n,T)$ as above. This proves 
\eqref{eq:hyp_tightness}.

\begin{lemma}\label{le:concdist}
  Let $R:=(X_{i1},\ldots,X_{in})$ denote a row of the matrix $X$ and assume
  \eqref{eq:hyp_tightness1}. Let $\psi:\dN\to\dN$ be such that
  $\psi(n)\to\infty$,
  and $\psi(n)<n$.
  There exists $\veps>0$ such that for any subspace $H$ of $\dC^n$
  with $1 \leq \DIM(H) \leq n - \psi(n)$, one has
  \begin{equation}\label{lemmsi1}
    \dP\PAR{\DIST(R,H) \leq \veps\si\sqrt{n-\DIM(H)}} %
    \leq e^{-\veps   \si^2  \psi(n)} + e^{-\veps   \psi(n)^2  / n }.
  \end{equation}
\end{lemma}

\begin{proof}
  As in \cite[Proposition 5.1]{tao-vu-cirlaw-bis} we can assume that the
  random variables $X_{ij}$ are centered, since this amounts to replace $H$
  with $H'$, the linear span of $H$ and a deterministic one-dimensional space,
  satisfying $ \DIM(H')\leq n-\psi(n)+1$. Next, we truncate $X_{ij}$: Fix
  $T>0$ and use Chebyshev's inequality to bound $\dP(|X_{1j}|>T)\leq
  T^{-2}\si^2$. Let $E_n$ denote the event that $\sum_{j=1}^n\IND(|X_{1j}|>T))
  \geq c\psi(n)$, for some $c>0$ to be chosen later. We take $T$ such that
  $T^2\geq 4n\si^2/(c\,\psi(n))$. Then, from Hoeffding's inequality one has,
  \begin{equation}\label{psinq}
    \dP\PAR{E_n}\leq e^{- c^2  \psi(n)^2/n}.
  \end{equation}
  Thus, in proving \eqref{lemmsi1} we can now assume that the complementary
  event $E_n^c$ occurs. Set $T=\max\{2\si \sqrt{n/(c\psi(n))},a\}$, where
  $a>0$ is the parameter in \eqref{eq:hyp_tightness}. Conditioning on the set
  $I$ such that $|X_{1j}|\leq T$ iff $j\in I$, $|I|\geq n-c\,\psi(n)$, and
  conditioning on the values of $X_{1j}$ for $j\in I^c$, one can reduce the
  problem to estimating $ \DIST(R,H')$ by $\DIST(\tilde R,H'')$ where $\tilde
  R$ is the vector $\tilde X_{1j}$, $j=1,\dots,|I|$, made of i.i.d.\  copies of the
  centered variable $\bx^{(T)} - \dE\bx^{(T)}$, and $H''$ has dimension at
  most $\DIM(H')+|I^c| + 1\leq n-\psi(n)+2+c\,\psi(n)\leq n-(1-2c)\psi(n)$. A
  simple computation yields
  \[
  \dE\SBRA{\DIST(\tilde R,H'')^2} %
  =\tilde \si^2 (n-\DIM(H''))\,,\qquad \tilde \si^2 = \VAR( \bx^{(T)}).
  \]
  Using \eqref{eq:hyp_tightness}, one has $\tilde\si^2\geq c_1\si^2$ for
  some constant $c_1>0$. As in \cite{tao-vu-cirlaw-bis} we may now invoke
  Talagrand's concentration inequality for Lipschitz convex functions of
  bounded independent variables. Using $n-\DIM(H'')\geq (1-2c)\psi(n)$, if $c$
  is sufficiently small this implies that for some $\veps>0$ one has
  \begin{equation}\label{lemmsi10}
    \dP\PAR{\DIST(\tilde R,H'') \leq \veps\si\sqrt{n-\DIM(H)}} %
    \leq \exp(-\veps \si^2T^{-2} \psi(n)).
  \end{equation}
  The expression above is then an upper bound for the probability of the
  event \[E_n^c\cap\big\{\DIST(R,H) \leq \veps\si\sqrt
  {n-\DIM(H)}\big\}.\] 
  Therefore if $T$ is bounded one has the upper bound
  $e^{-\veps'\si^2\psi(n)}$, for some $\veps'>0$, while if $T\to\infty$, then
  $\si^2T^{-2}=c\psi(n)/(4n)$ and one has the upper bound $e^{-\veps'
    \psi(n)^2/n}$. This ends the proof of \eqref{lemmsi1}. 
\end{proof}

\subsubsection{Compressible and incompressible vectors}
For $\delta \in( 0,1)$, define the set of \emph{sparse vectors}
\[
\SPARSE(\de):=
\{x\in\dC^n:\;|\SUPP(x)|\leq \de n\}
\]
where $\SUPP(x)=\{i:\,x_i\neq 0\}$ and $|\SUPP(x)|$ is its cardinality. Given
$\rho\in( 0,1)$, consider the partition of the unit sphere $\dS^{n-1}$ into a
set of \emph{compressible vectors} and the complementary set of
\emph{incompressible vectors} as follows:
\begin{gather*}
  \COMP(\de,\rho):= \{x\in\dS^{n-1}:\DIST(x,\SPARSE(\de))\leq\rho\}
  \\
  \INCOMP(\de,\rho):= \dS^{n-1}\setminus\COMP(\de,\rho).
\end{gather*}
For any matrix  $A \in \cM_n (\dC)$:
\begin{equation}\label{eq:decompsn}
  s_n ( A) = \min_{ x \in \dS^{n-1}} \NRM{A x}_2 %
  = \min\PAR{\min_{x\in\COMP(\de,\rho)}\NRM{A x}_2,\min_{x\in\INCOMP(\de,\rho)}\NRM{Ax}_2}.
\end{equation}
We will apply \eqref{eq:decompsn} to $A=Y^\top$, the transpose of the matrix 
\begin{equation}\label{matrY}
  Y=\si\sqrt{n}(M(n)-z),
\end{equation} 
and then use the obvious identities
$s_i(Y^\top)=s_i(Y)=\si\sqrt{n}\,s_i(M-z)$, $i=1,\dots, n$.

Next, we recall two lemmas from \cite{MR2407948}; see also \cite[Appendix
A]{bordenave-chafai-changchun}.

\begin{lemma}
  \label{le:incompspread}
  Let $x \in \INCOMP(\de,\rho)$. There exists a subset $\pi \subset \{1, \ldots,
  n\}$ such that $|\pi | \geq \de n / 2$, and for all $i \in \pi$,
  \[
  \frac{\rho}{\sqrt{n}} \leq |x_i|\leq \sqrt{\frac{2}{\delta n}}.
  \]
\end{lemma}

\begin{lemma}
  \label{le:invIncomp}
  Let $A\in \cM_{n} (\dC)$ be any random matrix and let $W_k$ denote its
  $k$-th column. For $1\leq k\leq n$, let $H_k=\SPAN\{W_j,\;j\neq k\}$. Then,
  for any $t \geq 0$,
  \[
  \dP\PAR{%
    \min_{ x \in \INCOMP(\de,\rho)} \NRM{Ax}_2 %
    \leq \frac{t \rho}{\sqrt{n}}} %
  \leq \frac{2}{\de n} \sum_{k=1}^n \dP \PAR{\DIST(W_k , H_k) \leq t}.
  \]
\end{lemma}

\subsubsection{Small ball probabilities and related estimates}
We turn to some crucial estimates. We assume that the hypothesis of theorem
\ref{th:eigenvals} hold.
\begin{lemma}
  \label{le:smallball}
  Let $R:=(X_{i1},\ldots,X_{in})$ denote a row of the matrix $X$. There exists
  $C>0$, such that for any $t\geq 0$, any $\de,\rho\in(0,1)$, and
  $x\in\INCOMP(\de,\rho)$, and any $w\in \dC$,
  \begin{equation}\label{eq:etaRw}
    \dP \PAR{|\ANG{R,x} - w| \leq t } %
    \leq \frac{C}{\si \sqrt{\delta n}}\PAR{\frac{t\sqrt{n}}\rho+1 }.
  \end{equation}
\end{lemma}

\begin{proof}
  Note that the left hand side is bounded above by $p_x(t)$, where $p_x$ is
  the concentration function defined in \eqref{concfunc}. By lemma
  \ref{le:incompspread}, we can assume that for all $1 \leq i \leq \delta n /
  2$, $ | x_i | \geq\rho / \sqrt{n}$. It suffices to invoke theorem
  \ref{th:halasz} with $n$ replaced by $\lfloor \delta n / 2 \rfloor $, with
  $x$ replaced by $y = x \sqrt{n}/\rho $ and use the identity: $p_{\gamma y}
  (\gamma t) = p_y ( t)$ for all $\gamma>0$.
\end{proof}

 \begin{lemma}
   \label{le:smallball2}
   Let $s(n)=n^\kappa$ for some $\kappa>0$. There exists $\veps>0$ 
   such that, if
   \begin{equation}\label{derho}
     \de=\frac{\si^2(n)}{\log(n)}
     \,,\quad 
     \rho=\frac{\veps\si(n)}{s(n)\sqrt{\de}}\,, 
   \end{equation}
   then for all $n$ large enough:
   \begin{equation}\label{eq:invcomp-star}
     \dP\PAR{%
       \min_{x\in\COMP(\de,\rho)}\NRM{Y^\top x}_2 %
       \leq \frac{ \veps \si (n)}{ \sqrt{\delta}} %
       \ ;\ s_1(Y) \leq s(n)  } %
     \leq  \exp(-\veps n\si^2(n) ).
   \end{equation}
 \end{lemma}

 \begin{proof}
   If $A \in \cM_n (\dC)$ and $y \in \dC^n$ is such that $\SUPP(y)\subset\pi
   \subset \{1, \ldots , n\}$, then
   \[
   \NRM{A y}_2\geq \NRM{y}_2s_n(A_{|\pi}),
   \]
   where $A_{| \pi}$ is the $n \times |\pi|$ matrix formed by the columns of $A$
   selected by $\pi$. Therefore,
   \begin{equation}\label{eq:comp}
     \min_{ x \in \COMP(\de,\rho)} \NRM{Ax}_2 %
     \geq (1-\rho) %
     \min_{\pi
       : \,|\pi| = \FLOOR{\delta n}}s_n(A_{|\pi}) %
     - \rho s_1 (A).
   \end{equation}
   On the other hand, for any $x \in \dC^{|\pi|}$,
   \begin{align*}
     \NRM{A_{|\pi} x}_2 ^2 %
     &=\Big\|\sum_{i \in \pi} x_i W_i\Big\|_2^2 \\ %
     &\geq\max_{i \in \pi}|x_i|^2 \DIST^2 (W_i , H_{i}) \\ %
     &\geq\min_{i \in \pi}\DIST^2(W_i,H_{i})\,\frac{1}{|\pi|}\sum_{j\in\pi}|x_j|^2,
   \end{align*}
   where $W_i$ is the $i$-th column of $A$ and $H_i:=\SPAN\{W_j:j \in \pi, j \ne
   i\}$. In particular,
   \begin{equation}\label{eq:sn-pi-dist-h}
     s_n ( A_{|\pi}) %
     \geq \frac1{\sqrt{|\pi|}}\,\min_{i \in \pi}  \DIST(W_i, H_i) . 
   \end{equation}
   Next, we apply \eqref{eq:sn-pi-dist-h} to $A = Y^\top$. We have
   $W_i=R^\top_i+(- D_{ii}+mn- z\si\sqrt{n})e_i$ where $R_i$ is the $i$-th row
   of $X$. Therefore,
   \[
   \DIST(W_i,H_i)\geq \DIST(R_i,\SPAN\{H_i,e_i\}) =  \DIST(R_i,H_i')
   \]
   where $H_i'=\SPAN\{H_i,e_i\} = \SPAN\{R_j, e_i :j \in \pi, j\ne i \}$. $H_i
   '$ has dimension at most $n\delta + 1$ and is independent of $R_i$. By
   lemma \ref{le:concdist}, with e.g.\ $\psi(n)=n-2\de n\geq n/2$, one has
   that, for some $\veps>0$
   \[
   \dP\PAR{\min_{i \in \pi}  \DIST(W_i, H_{i}) \leq \veps  \si \sqrt{n}}
   \leq n \exp ( - \veps\si^2 n )
   \]
   for all $n$ large enough. From \eqref{eq:sn-pi-dist-h}, for $|\pi|\leq \de
   n$:
   \[
   \dP \PAR{s_n((Y^\top)_{|\pi})\leq\frac{\veps\si}{\sqrt\delta}} %
   \leq n\exp(-\veps\si^2n).
   \]
   Therefore, using the union bound and $1-\rho\geq 3/4$, we deduce from
   \eqref{eq:comp}
   \[
   \dP\PAR{%
     \min_{x\in\COMP}\NRM{Y^\top x}_2\leq\frac{ \veps \si }{ 2\sqrt{\delta}} %
     \ ;\ s_1(Y) \leq s(n) } %
   \leq {\binom{n}{\FLOOR{\delta n}}} n e^{- \veps \si^2 n} %
   = e^{n( h(\delta ) (1+o (1) ) -\veps\si^2 )},
   \]
   with $h(\delta) := - \de \log \de - (1- \de) \log (1 -\de)$. As
   $n\to\infty$, $h(\de)=-\de\log \de(1+o(1))$ and using $n\si^2\gg\log(n)$ one
   has $-\de\log\de\leq \veps\si^2/2$ for all $n$ large enough. Therefore,
   \eqref{eq:invcomp-star} follows by adjusting the value of $\veps$.
\end{proof}

 \begin{lemma}\label{le:W_k,H_k}
   Let $W_k,H_k$ be as in lemma \ref{le:invIncomp}, with $A=Y^\top$. Let
   $\de,\rho,s(n)$ be as in lemma \ref{le:smallball2}. There exists $C>0$ such
   that, for all $1\leq k\leq n$ and $t \geq 0$,
   \begin{equation}\label{eq:invincomp}
     \dP\PAR{\DIST(W_k,H_k) \leq t \,;\, s_1 (Y) \leq s(n)} %
     \leq \frac{C}{\si\sqrt{\delta n}}\PAR{\frac{t\,s(n)}{\rho\,\si|\al|}+1},
   \end{equation}
   where $\al= m\sqrt{n}/\si-z $. 
 \end{lemma}
 Note that the $z$ dependence, represented by the coefficient $\al$ in the
 above estimate, cannot be completely avoided since if $z=m\sqrt{n}/\si$ then
 $s_n(Y)=0$.

 The proof of lemma \ref{le:W_k,H_k} requires a couple of intermediate steps.
 Fix $k$ and, conditional on $H_k$, consider a unit vector $\zeta$ orthogonal
 to $H_k$. Since $W_k$ is independent of $H_k$, the random vector $\zeta$ can
 be assumed to be independent of $W_k$. Clearly,
 \begin{equation}\label{eq:distnormal}
   \DIST(W_k,H_k) \geq |\ANG{\zeta,W_k}|.
 \end{equation}
 Define $\phi=(1,\ldots,1)$ and $\Phi = \SPAN \{
 \phi \} = \{\la\phi,\,\la\in\dC\}$. 
 
 \begin{lemma}
   \label{le:zetaphi}
   Let $n \geq 2$. The unit vector $\zeta$ orthogonal to $H_k$ satisfies
   \begin{equation}\label{eq:1/2}
     \DIST(\zeta,\Phi) \geq \frac{|\al|\si \sqrt{n}}{2 s_1(Y)},
   \end{equation}
   where $\al= m\sqrt{n} / \si-z$. In particular, if $s_1(Y)\leq s $, then
   one has, for any $\la\in\dC$,
   \begin{equation*}\label{eq:1/20}
     \|\zeta-\la\phi\|\geq  \frac{|\al|\si \sqrt{n}}{2 s }.
   \end{equation*}
 \end{lemma}
 
 \begin{proof}
   Set $\hat\phi=\frac{1}{\sqrt{n}}\phi$. Since $\| \zeta \| = \|\hat \phi \|=
   1$, we have
   \[
   \DIST( \zeta,\Phi) = \DIST(\hat \phi, \SPAN\{\zeta \}) = \DIST(\hat \phi,
   \SPAN\{\bar\zeta \}).
   \]
   (the last identity follows from $\hat \phi \in \dR^n$). Let $B \in \cM_{n}
   (\dC)$ be the matrix obtained from $Y$ by replacing the $k$-th row with the
   zero vector. Then, by construction $B \bar \zeta = 0$. Hence, $\bar \zeta
   \in \ker B$ and
   \begin{equation*}\DIST( \zeta,\Phi) \geq \DIST(
     \hat \phi,\ker B).
   \end{equation*}
   Observe that $\hat\phi=au+bv$, where $a,b\in\dC$, and some unit vectors
   $v\in \ker B$ and $u\in (\ker B)^\perp$. Then $|a|=\DIST(\hat\phi,\ker B)$
   can be bounded as follows. Note that $Y$ satisfies
   $Y\phi=\al\si\sqrt{n}\phi$ and therefore
   $aBu=B\hat\phi=\al\si\sqrt{n}\hat\phi-\al\si e_k$. Consequently, one has
   \[
   s_1(B)|a|\geq \|aBu\| %
   = \|B\hat \phi\| = | \al |\si \sqrt{n-1} \geq | \al |\si \sqrt{n} /2
   \]
   This implies \eqref{eq:1/2} since $s_1(B)\leq s_1(Y)$.
 \end{proof}
 
 \begin{lemma}
   \label{le:normalincomp}
   Let $\de,\rho,s(n)$ be as in lemma \ref{le:smallball2}. There exists
   $\veps>0$ such that, for all $n$ large enough:
   \[
   \dP\PAR{ \exists \la \in \dC : \eta(\la) \in \COMP(\de,\rho)\, ; s_1 (Y) \leq s(n) } %
   \leq \exp ( -\veps\si^2 n),
   \]
   where for $\la \in \dC$, $\eta (\la) :=(\zeta-\la
   \phi)/\|\zeta-\la \phi\|$.
 \end{lemma}
 
 \begin{proof}
   Let $B$ be as in the proof of lemma \ref{le:zetaphi}. Thus $B \bar\zeta =
   0$. Note that by lemma \ref{le:zetaphi}, $\eta(\la)$ is well defined as
   soon as $\al \neq 0$, and then $\|\eta (\la)\|=1$. If there exists
   $\la \in \dC$ such that $\eta (\la) \in \COMP(\de,\rho)$, then
   $B(\bar \eta (\la) +\la' \phi)=0$ for $\la ' = \bar\la /
   \|\zeta-\la \phi\| \in\dC$. Therefore
   \[
   \min_{ x \in \COMP(\de,\rho),\,\la\in\dC} \NRM{B (x+\la\phi)} =0.%
   \]
   Note that if $\Phi_k :=\SPAN\{\phi,e_k\}$, then $Y\phi\in\Phi, B \phi\in
   \Phi_k$ and $Bx-Yx\in\Phi_k$ for all vectors $x$. Thus
   \[
   \min_{ x \in \COMP(\de,\rho),\, v\in\Phi_k} \NRM{Y x+v}_2 =0.%
   \]
   The above equation can be rewritten as
   \begin{equation}\label{eq:pg}
     \min_{ x \in \COMP(\de,\rho)} \NRM{\Pi Y x}_2 =0,
   \end{equation}
   where $\Pi $ is the orthogonal projection on the orthogonal complement of
   $\Phi_k$. On the other hand, one has
   \[
   \NRM{\Pi Y x}_2 \geq \NRM{x}_2s_n([\Pi Y]_{|\pi}),
   \]
   if $\pi$ is the support of $x$. As in \eqref{eq:sn-pi-dist-h} one
   has
   \[
   s_n ([\Pi Y]_{|\pi}) \geq \frac1{\sqrt{ |\pi|}}\,\min_{i \in \pi}
   \DIST(\tilde W_i, \tilde H_i),
   \]
   where $\tilde W_i$ is the $i$-th column of $\Pi Y$ and $\tilde
   H_i:=\SPAN\{\tilde W_j:j \in \pi, j \ne i\}$. Note that \[\DIST(\tilde W_i,
   \tilde H_i)\geq \DIST(C_i, H_i''),\] where $H_i''=\SPAN\{C_j, e_\ell,
   \Phi_k:j, \ell \in \pi, j \ne i\}$ and $C_i$ is the $i$-th column of $X$.
   Then, since $H_i''$ has dimension at most $1+2n\delta$ and is independent
   of $C_i$, by lemma \ref{le:concdist}, with e.g.\ $\psi(n)=n-3\de
   n$,
   \[
   \dP \PAR{ s_n ( [\Pi Y]_{|\pi}) %
     \leq \frac{ \veps \si}{\sqrt{\delta}}} %
   \leq n \exp (-\veps \si^2 n),
   \]
   for some $\veps>0$. In particular, as in the proof of
   \eqref{eq:invcomp-star}, one sees that the probability of the event in
   \eqref{eq:pg} intersected with the event $s_1(Y)\leq s(n) $ is bounded by
   $\exp ( - \veps \si^2 n)$, for some new $\veps>0$ and all $n$ large enough.
   This proves the lemma. 
 \end{proof}

 Let us now go back to \eqref{eq:distnormal}. Observe that
 \begin{equation*}\label{eq:distnormal2}
   \ANG{\zeta,W_k}=\sum_{i=1}^n(\bar \zeta_i-\bar \zeta_k)X_{ki}+\ANG{\zeta,v}
   = \ANG{\zeta-\zeta_k\phi,R_k}+\ANG{\zeta,v}, 
 \end{equation*}
 where $v$ is a deterministic vector, and we use the notation $R_k$ for the
 $k$-th row of $X$. With the notation of lemma \ref{le:zetaphi} and lemma 
 \ref{le:normalincomp}, on the event $s_1(Y)\leq s$ we can write
 \begin{equation}\label{eq:distnormal3}
   \DIST(W_k,H_k)   \geq   \frac{|\al|\si \sqrt{n}}{2 s } \,
   |\ANG{\eta,R_k} + w_k|  ,
 \end{equation}
 where $\eta = \eta (\la)$, at $\la=\zeta_k$, and $w_k$ depends only on
 $\zeta$ and therefore is independent of $R_k$. By lemma \ref{le:normalincomp}
 and using $\exp ( -\veps\si^2 n)\ll (\si\sqrt{\de n})^{-1}$, in order to
 prove \eqref{eq:invincomp}, it is sufficient to invoke lemma
 \ref{le:smallball} above. This ends the proof of lemma \ref{le:W_k,H_k}.

 \subsubsection{Proof of Proposition \ref{prop:small}
 }\label{proof:small}
 Take $\de,\rho$ and $s(n)$ as in lemma \ref{le:smallball2}. From
 \eqref{eq:invincomp} and lemma \ref{le:invIncomp} we find, for all $t \geq
 0$,
  \begin{equation*}
  \dP\PAR{\min_{x\in\INCOMP(\de,\rho)}\NRM{Y^\top x}_2 %
    \leq \frac{\rho^2 t }{\sqrt{n}}\ ;\ s_1 ( Y) \leq s }  %
  \leq \frac{C}{\si \sqrt{ \delta^ 3 n } } \PAR{\frac{t s}{|\al|\si }+1 }.
\end{equation*}
Using our choice of $\rho,\delta, s(n)$, we obtain for some new constant $C>0$, for all $t \geq 0$,
\begin{equation}\label{eq:finalincomp}
  \dP\PAR{%
    \min_{x\in\INCOMP(\de,\rho)}\NRM{ Y^\top  x}_2 %
    \leq t %
    \ ;\ s_1(Y) \leq n^\kappa } %
  \leq  \frac{ C (\log(n))^{3/2}} {\si^4\sqrt{n}}  %
  \PAR{\frac{tn^{3\kappa}\sqrt{n}}{|\al|\si\log(n)   }+ 1 }.
\end{equation}
From
\eqref{eq:invcomp-star} we know that
\begin{equation}\label{eq:finalcomp}
  \dP\PAR{%
    \min_{x\in\COMP(\de,\rho)}\NRM{Y^\top x}_2\leq \veps \sqrt{\log(n)} %
    \ ;\ s_1(Y) \leq n^\kappa  } %
  \leq  \exp(-\veps n\si^2(n) ).
\end{equation}
Since $s_1(Y)=\si\sqrt{n}\|M(n)-z\|$, by \eqref{s1bound} one has
\begin{equation}\label{s1ybo}
  \dP\PAR{\,s_1(Y)> s } \leq Cn^2s^{-2}.
\end{equation}
Suppose that \begin{equation}\label{nonspa1}
  \si^2\geq n^{-1/4+\veps},
\end{equation} 
for some $\veps>0$. Choosing e.g.\ $\kappa=1+\ga$, $s=n^{1+\ga}$, with
$\ga=0.1$, and $t=n^{-4}$, then the above expressions and \eqref{eq:decompsn}
imply that
\begin{equation}\label{nonspa2}
  \lim_{n\to\infty}\dP\PAR{\,s_n(Y)> n^{-4} } =0.
\end{equation}
In particular, this proves Proposition \ref{prop:small} under the mild sparsity assumption \eqref{nonspa1}.

In the general case we cannot count on \eqref{nonspa1}, and we only assume
$n\si^2\gg (\log(n))^{6}$. In particular, while the bound \eqref{eq:finalcomp}
is still meaningful, the bound \eqref{eq:finalincomp} becomes useless, even at
$t=0$, if e.g.\ $\si^2\leq n^{-1/4}$. To deal with this problem we use a
further partition of the set $\INCOMP(\de,\rho)$ inspired by the method of
G\"oetze and Tikhomirov \cite{gotze-tikhomirov-new}. More precisely, for fixed
$\veps>0,\kappa>1$, define
\begin{gather}
  \de_1=\frac{\si^2(n)}{\log(n)}\,,\quad
  \rho_1=\frac{\veps\si(n)}{n^\kappa\sqrt{\de_1}}\,,
  \nonumber\\
  \de_{\ell}=\de_{\ell-1}\log(n)\,,\quad
  \rho_{\ell}=\frac{\veps\sqrt{\de_{\ell-1}}}{n^\kappa}\,\rho_{\ell-1}\,, %
  \quad\ell=2,\dots,N,
  \label{derhoseq}
\end{gather}
where $N=N(n)$ is defined as the smallest integer $k\geq 2$ such that $(\log
n)^{k+1}\si^2(n)\geq 1$. Note that $\de_1,\rho_1$ are the choice of values of
$\de,\rho$ from lemma \ref{le:smallball}. Define further
\begin{equation}\label{deLrhoL}
  \de_{N+1}=\frac1{(\log(n))^2}\,,\;\;\text{ and }\;\; \rho_{N+1} %
  =\frac{\veps\sqrt{\de_{N}}\rho_{N}}{n^\kappa}\,.
\end{equation} 
It is immediate to check that this defines an increasing sequence
$\de_\ell=(\log(n))^{\ell-2}\si^2(n)$ and a decreasing sequence $\rho_\ell$,
$\ell=1,\dots,N+1$, such that $\si^2(n)(\log(n))^{-1}\leq \de_{\ell}\leq (\log
n)^{-2}$, $\rho_{N+1}\leq \rho_\ell\leq \veps(\log(n))^{1/2}n^{-\kappa}$, with
$N(n)=\cO(\log(n)/(\log\log(n)))$, and
 \begin{equation}\label{rhoL}
   \rho_{N+1}\geq \exp{(-\eta(\log(n))^2)},
\end{equation} 
for any $\eta>0$, for $n$ large enough. 

As in our previous argument, cf.\ \eqref{nonspa2}, the value of $\kappa$ is
not essential as long as $\kappa>1$, and it may be fixed for the rest of this
proof as e.g.\ $\kappa=2$. Next, using the sequences $(\de_\ell,\rho_\ell)$,
define
\begin{gather}\label{decompo}
  C_0 = \COMP ( \delta_1, \rho_1) \, , \;\;\; C_{N+1} %
  =   \INCOMP ( \delta_{N+1}, \rho_{N+1})\,,\nonumber\\
  C_\ell = \COMP ( \delta_{\ell+1} , \rho_{\ell+1} ) %
  \cap \INCOMP ( \delta_{\ell }, \rho_{\ell })\,,\;\;1 \leq \ell \leq N.
\end{gather}
Note that these sets form a partition $\dS^{n-1}=\cup_{\ell=0}^{N+1}C_\ell$
and therefore, for any $t\geq 0$:
\begin{equation}\label{eq:UnionPartition}
  \dP\PAR{s_n(Y)\leq t} %
  \leq  \sum_{\ell= 0}^{N+1}\dP\PAR{\min_{ x \in  C_\ell}\NRM{Y^\top x}_2\leq t}.
  \end{equation}
Thus, it will be sufficient to show that
\begin{equation}\label{eq:UnionPartition1}
  \lim_{n\to\infty}%
  \sum_{\ell= 0 }^{N+1}%
  \dP\PAR{\min_{x\in C_\ell}\NRM{Y^\top x}_2 \leq e^{-(\log(n))^2} } =0.
  \end{equation}
  For the term $C_0$, since $\de_1,\rho_1$ coincide with the choice of lemma
  \ref{le:smallball2}, one can use the estimate \eqref{eq:finalcomp} and
  \eqref{s1ybo}. For $\ell\geq 1$ we need a refinement of the argument used
  for \eqref{eq:finalincomp} and \eqref{eq:finalcomp}.
  \begin{lemma}\label{le:comp2}
    There exists a constant $\veps > 0$ such that if $\de_\ell,\rho_\ell$ are
    defined by \eqref{derhoseq}, then for all $\ell=1,\dots,N$:
    \begin{equation}\label{bobo1}
      \dP\PAR{ \min_{ x \in  C_{\ell}} \|  Y^\top x \|_2  %
        \leq \veps\rho_\ell\sqrt{\de_{\ell}} \,;\; %
        s_1(Y)\leq n^\kappa;\, %
        s_1(X)\leq n^\kappa} \leq \exp{\PAR{-\veps n (1\wedge\si^2\de_\ell n)}}.
  \end{equation}
\end{lemma}
\begin{proof}
  Let $x \in \INCOMP(\delta_\ell,\rho_\ell)$ and assume that $\SUPP( x )
  \subset \pi$ with $|\pi | =\de_{\ell+1}n$. Let $W_i = R_i ^\top + ( -
  D_{ii}+ mn - z\sqrt{n}) e_i $ be the $i$-th column of $Y^\top$ and $H =
  \SPAN ( e_i , i \in \pi)$, and write
  \begin{align}
    \|  Y^\top x \|_2 & = \| \sum_{i \in \pi} W_i x_i \|_2 \geq \DIST \Big(\sum_{i \in \pi} W_i x_i  , H \Big) \nonumber\\
    &= \DIST \Big(\sum_{i \in \pi} R_i x_i , H \Big) = \DIST \PAR{ X^\top x ,
      H } .
    \label{htk}
  \end{align}
  Next, we want to apply corollary \ref{cor3:halasz}. Taking $s=n^\kappa$,
  from \eqref{rhoL} one has that $\log ( \frac{ s} { \rho_\ell^2 \delta_\ell }
  )\ll (\log(n))^2$. Therefore, using $n\si^2(n)\gg(\log(n))^3$,
  $\de_{\ell+1}=(\log(n))^{\ell-1}\si^2(n)$, one has
  \begin{equation}\label{dell1}
    \de_{\ell+1}\leq \frac{  \eta ( 1 \wedge \si^2  \delta_\ell n ) } { \log ( \frac{ s} { \rho^2_\ell \delta_\ell } )} ,
  \end{equation}
  for any $\eta>0$, and all $n$ large enough. Hence, if $c_1$ is as in
  corollary \ref{cor3:halasz}, then
  \[
  \dP\PAR{\NRM{Y^\top x}_2\leq c_1\rho_\ell\sqrt{\de_\ell}\,;\, %
    s_1(X)\leq n^\kappa} %
  \leq \exp{ \PAR{- c_1 n ( 1 \wedge \si^2 \de_\ell n )} }.
  \]
  If $\eta \in(0, 1)$ and $B_H = H \cap \INCOMP(\de_\ell,\rho_\ell)$, there
  exists an $\eta$-net $N$ of $B_H$, of cardinality $( 3 / \eta )^{ k }$,
  $k=\de_{\ell+1}n$, such that
  \[
  \min_{ x : \DIST( x, B_H) \leq \eta } \| Y^\top x   \|_2 \geq  \min_{ x \in N } \| Y^\top x  \|_2 - \eta s_1(Y).
  \]
  Hence, if we take
  $\eta=\rho_{\ell+1}=\frac{\veps\rho_\ell\sqrt{\de_\ell}}{n^\kappa}$, from
  the union bound, and then using \eqref{dell1}:
  \begin{align*}
    &\dP \PAR{\min_{x:\DIST(x,B_H)\leq\rho_{\ell+1}}%
      \NRM{Y^\top x}_2\leq(c_1-\veps)\rho_\ell\sqrt{\de_\ell}\,;\, %
      s_1(Y)\leq n^\kappa;\, s_1(X)\leq n^\kappa} \\
    &\qquad
    \leq\exp{\PAR{\de_{\ell+1}n\log\PAR{\frac{6n^\kappa}{\rho_\ell\sqrt\delta_\ell}}%
        - c_1n(1\wedge\si^2\delta_\ell n)}} %
    \leq \exp{\PAR{-\frac{c_1}{2}n(1\wedge\si^2\delta_\ell n)}}.
  \end{align*}
  Finally, summing over all choices of $|\pi|=\de_{\ell+1} n$, one finds
  \begin{align*}
    &
    \dP \PAR{  \min_{ x \in  C_\ell} \|  Y^\top x \|_2   \leq  (c_1-\veps)  \rho_\ell   \sqrt{\de_\ell}\, ; \, s_1 ( Y ) \leq n^\kappa;\, s_1(X)\leq n^\kappa  }  \\
    &\qquad
    \leq  \exp{\PAR{ n h(\delta_{\ell+1}) ( 1 + o(1) ) - \frac{c_1}2 n ( 1 \wedge \si^2  \delta_\ell n )}  },
  \end{align*}
  with $h(\delta) = - \delta \log ( \delta) - ( 1 - \de) \log ( 1 - \de)$.
  Since $\de_{\ell+1}\log \de_{\ell+1}\ll ( 1 \wedge \si^2 \delta_\ell n )$,
  the above expression is bounded by $e^{- \frac{c_1}3 n ( 1 \wedge \si^2
    \delta_\ell n )}$ for all $n$ large enough. The conclusion follows by
  choosing e.g.\ $\veps=\frac{c_1}3$.
\end{proof}

Let us now conclude the proof of Proposition \ref{prop:small}. Observe that by
\eqref{s1ybo} we may assume that $s_1(Y)\leq n^\kappa$. Moreover, the same
argument proving \eqref{s1ybo} proves the same bound for $s_1(X)$. Thus, one
may assume that $s_1(X)\leq n^\kappa$ as well at the price of adding a
vanishing term to \eqref{eq:UnionPartition1}. Using \eqref{eq:finalcomp} (for
the case $\ell=0$) and \eqref{bobo1} (for the case $\ell=1,\dots,N$), together
with the simple bounds $\veps\rho_\ell\sqrt{\de_\ell}\gg e^{-(\log(n))^2}$, $n
( 1 \wedge \si^2 \delta_\ell n )\gg (\log
n)^2$, 
one has
\begin{equation}\label{htk1}
  \lim_{n\to\infty}\sum_{\ell= 0 }^{N}  \dP\PAR{ \min_{ x \in  C_\ell} \|  Y^\top x \|_2  \leq  
    e^{-(\log(n))^2} 
  } =0.
\end{equation}
Thus, to end the proof of \eqref{eq:UnionPartition1}, it remains to prove 
\begin{equation}\label{bobo2}
\dP\PAR{ \min_{ x \in  C_{N+1}} \|  Y^\top x \|_2  \leq  
 e^{-(\log(n))^2}
 } \to 0.
  \end{equation}
To prove \eqref{bobo2}, observe that  lemma \ref{le:invIncomp} and lemma \ref{le:zetaphi}, as in \eqref{eq:distnormal3},
 imply that for all $t\geq 0$:
\begin{equation*}\label{bobo3}
\dP\PAR{ \min_{ x \in  C_{N+1}} \|  Y^\top x \|_2  \leq  \frac{t\rho_{N+1}}{\sqrt{n}}
}
  \leq \frac{2}{\de_{N+1} n} \sum_{k=1}^n \dP \PAR{ |\ANG{\eta^{(k)},R_k} + w_k|\leq \frac{2ts_1(Y)}{|\al|\si \sqrt{n}}
  },
  \end{equation*}
  where $w_k\in\dC$ and $\eta^{(k)}\in \dS^{n-1}$ denote suitable random
  variables independent of $R_k$, the $k$-th row of $X$. Thanks to
  \eqref{s1ybo}, one can safely assume that $s_1(Y)\leq n^{\kappa}$. By
  exchangeability it is enough to consider the first row $R$ of $X$, and the
  associated random variables $\eta,w$. Using $\de_{N+1} =(\log(n))^{-2}$,
  taking $t=e^{-(\log(n))^2}\sqrt{n}/\rho_{N+1}$, and using $ 2n^\kappa
  e^{-(\log(n))^2} / ( |\al|\si \rho_{N+1} ) \leq e^{-\frac12(\log(n))^2} /
  |\al | $, for $n$ large, it is then sufficient to prove
\begin{equation}\label{bobo4}
  \lim_{n\to\infty}(\log(n))^2\,\dP \PAR{ |\ANG{\eta,R} + w|\leq |\al|^{-1}e^{-\frac12(\log(n))^2}
  }=0.
\end{equation}
By conditioning on the event $\eta\in \INCOMP(\de_{N+1},\rho_{N+1})$, lemma
\ref{le:smallball} implies that
\begin{align}
&\dP \PAR{ |\ANG{\eta,R} + w|\leq |\al|^{-1}e^{-\frac12(\log(n))^2};\eta\in \INCOMP(\de_{N+1},\rho_{N+1})}
  \nonumber\\
  &\qquad \quad
  \leq \frac{C}{\si \sqrt{\delta_{N+1}n}}
  \PAR{\frac{e^{-\frac12(\log(n))^2}{\sqrt{n}}}{|\al|\rho_{N+1}}+ 1  }\leq \frac{2C\log(n)}{\si \sqrt{n}},
\label{bobo5}
\end{align}
where the last bound holds for all $z\in\dC\setminus \La$, for $n$ sufficiently large, so that $\al$ is bounded away from $0$. Since by assumption \eqref{eq:hyp_sig2} we have $\si\sqrt{n}\gg (\log(n))^3$, this proves 
\eqref{bobo4}, provided that  
\begin{equation}\label{bobo6}
 \lim_{n\to\infty}(\log(n))^2\,
 \dP \PAR{\eta\in \COMP(\de_{N+1},\rho_{N+1}) 
 }=0.
  \end{equation}
As in the proof of lemma \ref{le:normalincomp}, cf.\ \eqref{eq:pg}, $\eta\in \COMP(\de_{N+1},\rho_{N+1})$ implies
\[\min_{ x \in \COMP(\de_{N+1},\rho_{N+1})} \NRM{\Pi Y x}_2 =0,\]
where $\Pi $ is the orthogonal projection on $(\SPAN\{\phi,e_1\})^\perp$.
Since $\COMP(\de_{N+1},\rho_{N+1})=\cup_{\ell=0}^NC_\ell$, \eqref{bobo6} may be reduced to the estimate
\begin{equation}\label{bobo7}
\lim_{n\to\infty}(\log(n))^2\sum_{\ell= 0 }^{N}  \dP\PAR{ \min_{ x \in  C_\ell} \|  \Pi Y x \|_2=0} 
 =0.
  \end{equation}
To prove \eqref{bobo7}, one repeats the argument in the proof of lemma \ref{le:comp2}. More precisely, \eqref{htk} 
is now replaced by 
\[
\| \Pi Y x \|_2  \geq \DIST \PAR{ X x , H' }
\]
where $H'=\SPAN\{H, \phi,e_1\}$. Since $H'$ has dimension at most
$\de_{\ell+1}n +2$, the same arguments apply here. As in the proof of
\eqref{htk1}, this implies \eqref{bobo7}. This concludes the proof of
Proposition \ref{prop:small}.

\subsection{Moderately small singular values}

\begin{lemma}[Moderately small singular values] \label{prop:Msmall} %
  Assume \eqref{eq:hyp_sig2} and \eqref{eq:hyp_tightness1}. Let $u(n)=n/[(\log
  n)^5]$. There exists $c_0$ such that for any $z \in \dC$,
  a.s.\ for $n \gg 1$
  \[
  s_{n - i } ( M(n) - z) \geq c_0 \frac{i}{n},\qquad  u(n) \leq i \leq n-1,
  \]
\end{lemma}
\begin{proof}
  We follow the original proof of Tao and Vu \cite{tao-vu-cirlaw-bis} for the
  circular law. To lighten the notations, we denote by $s_1\geq\cdots\geq s_n$
  the singular values of $M-zI$. We fix $u(n) \leq i \leq n-1$, and consider
  the matrix $Y'$ formed by the first $m:=n- \CEIL{i/2}$ rows of $\si
  \sqrt{n}(M-z)$. Let $s_1'\geq\cdots\geq s_m'$ be the singular values of
  $Y'$. By the Cauchy-Poincar\'e interlacing, we get
  \[
  \si^{-1}n ^{-1/2} s'_{n-i} \leq  s_{n-i}
  \]
  (see e.g.\ \cite[corollary 3.1.3]{MR1288752}).  Next, by \cite[lemma A4]{tao-vu-cirlaw-bis}, we have
  \[
  s'^{-2}_1  + \cdots + s'^{-2}_{n -  \CEIL{i/2}} %
  = \DIST_1^{-2}+\cdots+\DIST_{n -\CEIL{i/2}}^{-2}, 
  \]
  where $\DIST_j:=\DIST(R'_j,H_j')$ is the distance from the
  $j^\text{th}$ row $R'_j$ of the matrix $Y'$ to $H'_j$, the subspace spanned by all other
  rows of $Y'$. In particular, we have
  \begin{equation}\label{eq:stodist}
    \frac{i}{2n} s^{-2}_{n-i} 
    \leq \frac{i\si^2 }{2}s'^{-2}_{n-i}
    \leq  \si^2\sum_{j=n-\CEIL{i}}^{n-\CEIL{i/2}}s_{j}'^{-2}
    \leq \si^2 \sum_{j=1}^{n-\CEIL{i/2}}\DIST_{j}^{-2}. 
  \end{equation}
  Now, we note that 
  \[
  \DIST_j =\DIST(R'_j,H'_j) \geq \DIST ( R_j , H_j ),
  \]
  where $H_j = \SPAN \{ H'_j , e_j \}$ and $R_j$ is the $j^\text{th}$ row of
  $X$. 
  Now, $H_j$ is independent of $R_j$ and $\DIM(H_j)\leq n-\frac{i}{2} +1
  \leq n-\frac14\,u(n)$. 
  We may use  lemma \ref{le:concdist} with the choice $\psi(n)=\frac14\,u(n)$. 
  By assumption \eqref{eq:hyp_sig2} 
  one has
  \begin{equation}\label{v(n)}
    \min\{\si^2(n)\psi(n),\psi(n)^2/n\}\gg \log(n).
  \end{equation}
  By the union bound, this implies 
  \begin{equation}\label{lemmsi}
    \sum_{n\geq 1}\dP\PAR{\bigcup_{i=u(n)}^{n-1}\bigcup_{j=1}^{n-\CEIL{i/2}}
      \BRA{\DIST_j %
        \leq\frac{\si \sqrt{i}}{2\sqrt{2}}} }<\infty.
  \end{equation}
  Consequently, by the first Borel-Cantelli lemma, we obtain that a.s.\ for
  $n\gg1$, all $u(n)\leq i \leq n-1$, and all $1 \leq j \leq n - \CEIL{i/2}$,
  \[
  \DIST_j %
  \geq \frac{\si \sqrt{i}}{2\sqrt{2}} %
  \geq \frac{\si \sqrt{i}}{4}
  \]
  Finally, \eqref{eq:stodist} gives $s^{2}_{n-i}\geq (i^2)/(32n^2)$, i.e.\ the
  desired result with $c_0 := 1/(4\sqrt{2})$.
\end{proof}

\subsection{Proof of  theorem \ref{th:unifint}}
\label{se:proof:th:unifint}
Let us choose $J(t)=t^2$. By lemma \ref{le:large}, it is sufficient to prove
that
\[
\lim_{t \to \infty}\limsup_{n \to \infty}%
\dP\PAR{\int_0^1\!J(|\log s|)\,d\nu_{M(n)-z}(s) >t}=0.
\]
We shall actually prove that if $n_*$ is the last $i$ such that $s_{n-i}(M - z
) \leq 1$, then there exists $C > 0$ such that
\begin{equation}\label{eq:UIprob}
  \lim_{n \to \infty} %
  \dP\PAR{\frac{1}{n}\sum_{i =0 }^{n_*}J(|\log s_{n-i}(M-z)|)\leq C} = 1.
\end{equation}
With the notation, of lemma \ref{prop:Msmall}, let $F_n$ be the event, that
$s_{n} ( M - z) \geq e^{-(\log(n))^2}$ and that for all $u(n) \leq i \leq n-1$,
$s_{n - i } ( M- z) \geq c_0\, i / n $. Then by Proposition \ref{prop:small}
and lemma \ref{prop:Msmall}, $F_n$ has probability tending to $1$. Also, if
$F_n$ holds, writing $s_{n - i }$ for $s_{n - i } ( M- z)$ one has
\begin{align*}
  \frac{1}{n} \sum_{i =1 }^{n_*} J(|\log s_{n-i}|) & \leq \frac{1}{n} \sum_{i
    =1}^{u(n)} J((\log(n))^2) %
  + \frac{1}{n}\sum_{i = 1}^n  J(|\log(c_0\,i/n)|) \\
  & = \frac{u(n)(\log(n))^4}{n}+ \frac{1}{n} \sum_{i =1}^n
  (\log (n/i))^{2}.
\end{align*}
This last expression is uniformly bounded since $u(n)=n/[(\log(n))^5]$ and the
sum is approximated by a finite integral. This concludes the proof of
\eqref{eq:UIprob}.

\section{Limiting distribution: Proof of theorems \ref{th:naturemu} and
  \ref{th:propmu}}
\label{se:proof:th:naturemu+th:propmu}

\subsection{Brown measure}
\label{subsec:Brown}

In this paragraph, we recall classical notions of operator algebra. Consider the pair $(\cM,
\tau)$, where $\cM$ is a von Neumann algebra and $\tau$ is a normal,
faithful, tracial state on $\cM$. For $a \in \cM$, set $|a|= \sqrt{a^*a}$.
For a self-adjoint element  $a\in \cM$, we denote by $\mu_a$ the spectral
measure of $a$, that is the unique probability measure on the real line
satisfying, for any $z \in \dC_+$,
\[
\tau ( ( a - z )^{-1}  ) = \int\!\frac{d\mu_a(t)}{t-z} = S_{\mu_a} (z)  . 
\] 
The Brown measure \cite{MR866489} of $a \in \cM$ is the probability measure
$\mu_a$ on $\dC$, which satisfies for almost all $z \in \dC$,
\[
\int\!\log|z-\la|\,d\mu_a(\la) =  \int\!\log(t)\,d\mu_{|a - z |}(t) 
\]
In distribution, it is given by  the formula 
\begin{equation} \label{eq:defbrown}
\mu_a = \frac{1}{2\pi} \Delta  \int\!\log(t)\,d\mu_{|a-z|}(t). 
\end{equation}
Our notation is consistent: firstly, if $a$ is self-adjoint, then the Brown
measure coincides with the spectral measure; secondly, if $\cM = \cM_n (\dC)$
and $\tau = \frac{1}{n}\TR$ is the normalized trace on $\cM_n (\dC)$, then the
Brown measure of $A$ is simply equal to $\mu_A=\frac1n\sum_{i=1}^n\delta_{\la_i(A)}$.

The $\star$-distribution of $a \in \cM$ is the collection of all its
$\star$-moments $\tau (a^{\veps_1} a^{\veps_2} \cdots a^{\veps_n})$ where
$a^{\veps_i}$ is either $a$ or $a^*$. The element $c \in \cM$ is circular if
it has the $\star$-distribution of $(s_1+is_2)/\sqrt{2}$ where $s_1$ and $s_2$
are free semi-circular variables. We refer to Voiculescu, Dykema and Nica
\cite{MR1217253} for a complete treatment of free non-commutative variables.

As explained in Haagerup and Schultz \cite{MR2339369}, it is possible to
extend these notions to unbounded operators. Let $\bar \cM$ be the
set of closed, densely defined operators $a$ affiliated with $\cM$ satisfying
\[
 \int\!\log(1+t)\,d\mu_{|a|}(t) < +\infty. 
\]
In particular the normal
operator $g$ in theorem \ref{th:naturemu} is an element of $ \bar \cM$. Also,
note that if $a \in \bar \cM$ and $z \in \dC$, then $a - z \in \bar \cM$.
For all $ a \in \bar \cM$, Haagerup and Schultz check that it is possible to
define the Brown measure by \eqref{eq:defbrown}. 

\subsection{Proof of theorem \ref{th:naturemu}}

From theorem \ref{th:eigenvals}, $\mu$ is given by the formula, in distribution,
\[
\mu = \frac{1}{2\pi} \Delta  \int\!\log(t)\,d\nu_{z}(t).
\]
Hence in view of \eqref{eq:defbrown}, the statement of theorem \ref{th:naturemu} will follow once we prove that for all $z \in \dC$, 
\[
\nu_{z} = \mu_{|c + g - z|}.
\]
To prove the latter identity, assume that $\bx$ has distribution $\cN(0,K)$, i.e.\
$(X_{ij})_{1 \leq i,j \leq n}$ are i.i.d.\ centered Gaussian variable with
covariance $K$, and $\si = 1$. Let $(G_i)_{i \geq 1}$ be an independent sequence of i.i.d.\ Gaussian random variables, $G_i\sim\cN(0,K)$. We define the diagonal matrix
$D'=\mathrm{diag}(G_1,\ldots,G_n)$, which is independent of $X$ and set $M' = \frac{1}{\sqrt{n}}X-D'$. The proof of theorem \ref{th:singvals} shows that $\nu_{M-z}$, $\nu_{M'-z}$ both converge a.s. to $\nu_z$. However, it is a consequence of  Capitaine and Casalis \cite[proposition
  5.1]{MCMC04} or Anderson, Guionnet and Zeitouni \cite[theorem 5.4.5]{AGZ}, that  $\dE \nu_{M'-z}$ converges weakly to $\mu_{|c + g - z|}$ (\cite[theorem 5.4.5]{AGZ} is stated for Wigner matrices but the result can be lifted to our case, see \cite[exercice 5.4.14]{AGZ}).

\subsection{Quaternionic resolvent}

Here we give another characterization of the Brown measure $\mu_{c+g}$. 
We use the same linearization procedure as in section \ref{linearization} to develop a 
quaternionic resolvent approach for the Brown measure. 
This approach was  introduced in the mathematical physics literature
\cite{FZ97,Gudowska-Nowak,rogers2010} for the analysis of non-hermitian random matrices; 
see also \cite{bordenave-caputo-chafai-heavygirko} and \cite[\S 4.6]{bordenave-chafai-changchun}. As above, we
consider the operator algebra $(\bar \cM, \tau)$ associated to the von Neumann
algebra $(\cM, \tau)$. If elements of $\cM$ act on a Hilbert space $H$, we define the Hilbert space $H_2 = H \times \dZ/2\dZ$
and for $x = (y,\veps) \in H_2$, we set $\hat x = ( y, \veps +1 )$. In
particular, this transform is an involution $\hat {\hat x} = x$. There is the
direct sum decomposition $ H_2 = H_0 \oplus H_1$ with $H_\veps = \{ x =
(y,\veps) : y \in H \}$. 
An operator $b$ acting on $H_2$ has the $2\times 2$ representation 
\begin{equation}\label{bij}
b=\begin{pmatrix} b_{00}   & b_{01}  \\  
 b_{10} & b_{11}
  \end{pmatrix}
\end{equation}
where $b_{ij}$ are operators on $H$.
That is,  if $x=(y_0,0)+(y_1,1)\in H_2$, then $bx = (z_0,0)+(z_1,1)$, where $z_0=b_{00}y_0+b_{01}y_1$, and $z_1=b_{10}y_0+b_{11}y_1$.
We define the linear map
$\tau_2$ on operators acting on $H_2$, with values in $\cM_2(\dC)$, through the formula, 
\begin{equation}\label{eq:deftau2}
  \tau_2 (b) =  
  \begin{pmatrix} \tau ( b_{00}  )  & \tau ( b_{01}  )  \\  
    \tau ( b_{10} ) & \tau (b_{11} )
  \end{pmatrix}.
\end{equation}
Given $a\in\bar \cM$ we define the operator 
\[
\BIP(a)=\begin{pmatrix} 0 & a  \\  
a^* & 0
  \end{pmatrix}
\]
The operator $\BIP(a)$ is self-adjoint.  
It will be called the {\em bipartization} of $a$. 

Recall the definition \eqref{H_+} of $\dH_+$ and 
$q = q(z, \eta) \in \dH_+$
We define the \emph{quaternionic
transform} of $a$ as the $2\times 2$ matrix 
\[
\Gamma_a (q) := \tau_2 ( (\BIP(a) -  q \otimes I_{H} )^{-1} ).
\]
Here $q \otimes I_{H}$ is the operator on $H_2$ defined by \eqref{bij} with $b_{00}=b_{11}=\eta I $
and $b_{01}=z I$, $b_{10}=\bar z I$, with $I$ the identity operator on $H$.
Note that $(\BIP(a) - q\otimes I_H )^{-1}$ is the usual resolvent at $\eta$ of
the self-adjoint operator $b (z): = \BIP(a) - q(z,0) \otimes I_H$. Hence $(\BIP(a) -
 q \otimes I_{H} )^{-1}$ inherits the usual properties of resolvent operators
(analyticity in $\eta\in\dC_+$, bounded norm). For a proof of the next lemma, see
\cite[lemma 4.19]{bordenave-chafai-changchun}. We use the notation $\pd=\frac12(\partial_x-i\partial_y)$, for the derivative 
at $z=x+iy\in\dC$.  

\begin{lemma}[Properties of the quaternionic transform]\label{le:propRes}
  For all $q = q (z,\eta)\in \dH_+$,
  \[
  \Gamma_a ( q )  %
  = \begin{pmatrix} 
    \al(q) & \beta (q) \\  
    \bar \beta (q)   &  \al(q) 
  \end{pmatrix} \in \dH_+, 
  \] 
  with
  \[
  \al(q) = S_{\check \mu_{ |a - z |} } (\eta) , 
  \]
  and, in distribution,  
  \[
  \mu_a = - \frac{1}{\pi} \lim_{t \downarrow 0} \pd  \beta (q ( z, it ) ). 
  \]
\end{lemma}

Recall that a sequence of matrices $(A_n)_{n \geq 1}$ is said to converge in
$\star$-moments to $a \in \cM$ if for any integer $k$ and $(\veps_i)_{1 \leq i
  \leq k} \in \{1, * \}^k$,
\[
\lim_{n \to \infty} %
\frac{1}{n} \TR A_n^{\veps_1} A_n^{\veps_2} \cdots A_n^{\veps_k} %
= \tau (a^{\veps_1} a^{\veps_2} \cdots  a^{\veps_k}).
\] 
Also, if  $(A_n)_{n \geq 1}$ is a sequence of random matrices, $(A_n)_{n \geq 1}$ converges in expected $\star$-moments to $a
\in \cM$ if the above convergence holds in expectation. 

\begin{lemma}[Continuity of the quaternionic transform]
  \label{le:starmoment}
  If $(A_n)$ is a sequence of matrices converging in $\star$-moments to $a \in
  \cM$ then for all $q \in \dH_+$, $\Gamma_{A_n}( q )$ converges to $\Gamma_a
  ( q )$. If $(A_n)$ is a sequence of random matrices converging in expected
  $\star$-moments to $a \in \cM$ then $\dE \Gamma_{A_n}( q )$ converges to
  $\Gamma_a ( q )$.
\end{lemma}

\begin{proof}
  For ease of notation, let $a_z = a - z$. Then, in matrix form,
  \[
  (\BIP(a) - q\otimes I_H )^{-1} = \begin{pmatrix} -\eta & a_z \\ a_z^* & -
    \eta \end{pmatrix}^{-1} = - \begin{pmatrix} \eta ( \eta ^2 - a_z a_z
    ^*)^{-1} & a_z ( \eta ^2 - a_z^* a_z)^{-1} \\ a_z^* ( \eta ^2 - a_z a_z
    ^*)^{-1} & \eta ( \eta ^2 - a_z^* a_z)^{-1} \end{pmatrix}.
  \]
  Hence
  \begin{equation*}\label{eq:Gammaa}
    \Gamma_a ( q ) =   - 
    \begin{pmatrix} \eta \tau ( \eta ^2 - a_z a_z ^*)^{-1}   
      &  \tau \PAR{ a_z ( \eta ^2 - a_z^* a_z)^{-1}}  \\ 
      \tau  \PAR{a_z^*  ( \eta ^2 - a_z a_z ^*)^{-1} }  
      &     \eta \tau ( \eta ^2 - a_z^* a_z)^{-1} 
    \end{pmatrix}.
  \end{equation*}
  For $a$ replaced by $A_n$ and $\tau$ by $\frac{1}{n}\TR$, we may expand in
  series the terms of the above expression. Each term of the series is a
  $\star$-moment of $(A_n - z)$ and it converges by assumption to the
  $\star$-moment in $a-z$. Since $|a|$ is bounded, for $|\eta |$ large enough,
  the series is absolutely convergent. We thus obtain the convergence of
  $\Gamma_{A_n}( q )$ for $| \eta|$ large enough. Finally, we may extend by
  analyticity to all $\eta\in\dC_+$.
  \end{proof}

If $\Gamma=\begin{pmatrix}\Gamma_{11}&\Gamma_{12}\\
\Gamma_{21}&\Gamma_{22}\end{pmatrix}$ is a $2\times 2$ matrix, we define  $ \DIAG(\Gamma)=\begin{pmatrix}\Gamma_{11}&0\\0&\Gamma_{22}\end{pmatrix}$.
\begin{proposition}[Subordination formula]
  \label{prop:subordination}
  If $c$ and $a$ are $\star$-free operators in $(\bar \cM, \tau)$ with $c$
  circular %
  and $a$ normal 
  then for all $q=q(z,\eta) \in \dH_+$,
  \begin{equation} \label{eq:FPGamma}
    \Gamma_{c+a}(q)=\Gamma_{|a-z|}\PAR{q(0,\eta)+\DIAG(\Gamma_{c+a}(q))}
    =\Gamma_{a}\PAR{q+\DIAG(\Gamma_{c+a}(q))}.
  \end{equation}
\end{proposition}

A version of proposition  \ref{prop:subordination}, in the language
of random matrices, was obtained by Rogers \cite[theorem 2]{rogers2010}. 
This subordination formula is also reminiscent of the subordination formula in
Biane \cite[proposition 2]{MR1488333} on the free sum with a semi-circular.
Our argument is indirect and relies on random matrices.

\begin{proof}[Proof of proposition \ref{prop:subordination}]
 Let $\mu_a$ be the spectral measure of the normal operator $a$. 
  From the spectral
  theorem,
  \begin{equation*}
    \Gamma_a(q) %
    = \int\!\PAR{\begin{pmatrix}0&x\\\bar x&0\end{pmatrix}-q}^{-1}\,d\mu_{a}(x).
  \end{equation*}

  We first assume that $a \in \cM$. Consider a diagonal matrix $A$ of size $n$
  with i.i.d. diagonal entries with distribution $\mu_a$ and $Y$ an independent complex
  Ginibre matrix of size $n$ (i.e. $(Y_{ij})_{1 \leq i , j \leq n}$ is an array of i.i.d. $\cN(0,I_2/2)$ random variables). The proof of theorem \ref{th:singvals}, cf.\ \eqref{eq:FPCircular}, shows
  that $\dE \Gamma_{A + Y/\sqrt{n}}$ converges to the function $\Gamma$ which satisfies
  \[
  \Gamma (q)= \Gamma_a \PAR{ q + \DIAG( \Gamma (q)) }.
  \]
  On the other hand, it is known that $A + Y/\sqrt{n}$ converges in expected
  $\star$-moments to the free sum $a+c$, see  \cite[proposition
  5.1]{MCMC04} or \cite[theorem 5.4.5]{AGZ}. It thus remains to invoke lemma \ref{le:starmoment}. This
  completes the proof of lemma \ref{prop:subordination} when $a \in \cM$.

  In the general case, let $a \in \bar \cM$ be a normal operator. From the
  spectral theorem, (see e.g. \cite[\S X.4]{conway90}), there is a resolution
  of the identity $E$ (i.e. a projection valued probability measure) such that
  \[
  a = \int_\dC \lambda d E (\lambda), 
  \]
  and $D(a) = \{ \psi \in H : \int |\lambda|^2 d \langle \psi , E (\lambda)
  \psi \rangle < \infty \}$. For $n$ integer, define
  \[
  a_n =
  \int_{\{ \lambda : | \lambda | \leq n \}} \lambda d E (\lambda) .
  \]
  By construction, as $n \to \infty$, for any $\psi \in D(a)$,
  \[
  \NRM{ a_n \psi - a \psi }^2_2  %
  =\int_{\{\la:|\la|> n\}}|\la|^2d\langle\psi , E(\la)\psi\rangle\to0. 
  \]
  That is, $a_n$ converges in strong sense toward $a$. Hence the sequences of
  operators, $\BIP ( a_n ) - q ( z , 0) \otimes I_H$ and $\BIP ( a_n + c ) - q
  ( z , 0) \otimes I_H $ also converge in the strong sense to $\BIP ( a ) - q
  ( z , 0) \otimes I_H$ and $\BIP ( a + c ) - q ( z , 0) \otimes I_H $,
  respectively. In particular, for any $q \in \dH_+$,
  \begin{equation}
    \label{eq:subor2}
    \Gamma_{a_n} (q)\to \Gamma_a (q) %
    \quad \text{and} \quad %
    \Gamma_{a_n + c} (q)\to \Gamma_{a + c} (q).  
  \end{equation}
  (see e.g.\ \cite[theorem VIII.25(a)]{reedsimon}). 
  
  Moreover, by construction $a_n$ is a bounded operator and, from
  what precedes, $\Gamma_{a_n + c}$ satisfies the fixed point equation
  \begin{equation}
    \label{eq:subor1} \Gamma(q) = \Gamma_{a_n} \PAR{ q +
      \DIAG( \Gamma(q) ) }.
  \end{equation}

We note finally that  
  \begin{equation}
    \label{eq:subor3} 
    \| \Gamma_a (q) - \Gamma_a (q') \| \leq C \| q - q' \|,
  \end{equation}
  where for $q = q ( z, \eta)$, $q' = q ( z',\eta')$, $C = \min ( \Im (\eta )
  , \Im ( \eta' ) )^{-2}$. Indeed, by the resolvent identity: 
  \[
  (b - q)^{-1} -
  (b - q')^{-1} = (b - q)^{-1} ( ( q' - q ) \otimes I_H ) (b - q')^{-1}.
  \]
  Hence 
  \[
  \NRM{\Gamma_a(q)-\Gamma_a(q')}%
  \leq\NRM{(b-q)^{-1}}\NRM{q-q'}\NRM{(b-q')^{-1}} %
  \leq C\NRM{q-q'}.
  \]
  The conclusion follows from
  \eqref{eq:subor2},\eqref{eq:subor1} and \eqref{eq:subor3}.
\end{proof}

\begin{remark}[Uniqueness of the solution to the fixed point equation]%
  \label{rk:proofnaturemu}
  Note that \eqref{eq:FPGamma} characterizes completely the quaternionic
  transform of $c+a$. Indeed, in section \ref{subsec:unicity}, we have proved
  that there exists a unique map $\Gamma : \dH_+ \to \dH_+$ which satisfies
  \eqref{eq:FPGamma} for all $q \in \dH_+$ and such that, with $\al (q) =
  \Gamma(q)_{11}$, for all $z \in \dC$, $\eta \mapsto \al ( q ( z, \eta )
  )$ is analytic on $\dC_+$ and is the Cauchy-Stieltjes transform of a
  symmetric measure on $\dR$. To see this,  in \eqref{eq:FPCircular}, replace
  $z-G$ by a random variable $G_z$ with law $\mu_{|a - z|}$ to obtain \eqref{eq:FPGamma}.
\end{remark}

\subsection{Proof of theorem \ref{th:propmu}}
 
Set \[\Gamma_{c + a}(q) = \begin{pmatrix} \al (q)& \beta(q) \\ \bar \beta(q) &
  \al(q) \end{pmatrix}.\] By proposition \ref{prop:subordination}, $\Gamma_{c+h} $ satisfies the fixed point equation
\begin{align}
  \begin{pmatrix} 
    \al  & \beta \\ 
    \bar \beta  &  \al  
  \end{pmatrix} 
  & = \dE\PAR{\begin{pmatrix} 0  & G \\ \bar G & 0  \end{pmatrix}   %
    - q - \begin{pmatrix}\al&0\\0&\al\end{pmatrix}}^{-1} \nonumber \\
  &=  \dE\frac{1}{|G-z|^2 -(\al+\eta)^2} %
  \begin{pmatrix} \al+\eta & G-z\\\bar G -\bar z&\al+\eta\end{pmatrix},
  \label{eq:fixpointgamma}
\end{align}
where $G$ has law $\cN(0,K)$ and $q = q(z,\eta)$. For ease of notation, we set
\[
G_z  = z - G
\] 
We also define
\[
\Si = \BRA{  z \in \dC : \dE\frac{1}{|G - z|^2}>1},
\]
and its closure $\bar \Si = \BRA{  z \in \dC : \dE\frac{1}{|G - z|^2}\geq1}$.
As in section \ref{subsec:unicity}, for $\eta = it$, we
find $\al = i h (z,t) \in i \dR_+$ and
\[
1 = \dE \frac{1+th^{-1}}{|G_z|^2+(h+t)^2}.
\]
In particular, if $0 < t \leq 1$, then
\[
1 \leq \dE \frac{1+h^{-1}}{|G_z|^2+h^2}.
\]
When $h$ goes to infinity, the right hand side goes to $0$ (uniformly in $z$).
Hence there exists $c >0$, such that for all $z \in \dC$ and $0 < t \leq 1$,
$h(z,t) \leq c$. Similarly,
\[
1 \geq \dE \frac{1}{|G_z|^2+(h+t)^2}.
\]
Thus for all $z \in \Si$ there exists $c_z > 0$ depending continuously on $z$
such that $h(z,t) + t \geq c_z$.

We now let $t \downarrow 0$. From what precedes, if $z \in \Si$, any
accumulation point, say $f(z)$, of $h(z,t)$ satisfies $f(z) \in [c_z, c]$ and
\begin{equation}\label{eq:FPfz}
1 =  \dE \frac{1}{|G_z|^2+f(z)^2}. 
\end{equation}
The function $\varphi_z : x \mapsto \dE ( |G_z|^2 + x^2)^{-1}$ is decreasing, for $x\geq 0$.
Hence, for all $z \in \Si$, there exists a unique value $f (z)$ which
satisfies \eqref{eq:FPfz}. Moreover, for any $0<\veps <1$ and $z \in \dC $,
the map $\varphi_z$ is $C^\infty$ on $[\veps, \veps^{-1}]$, while for any $x >
\veps$, the map $\psi_x : z \mapsto \dE ( |G_z|^2 + x^2)^{-1}$ is $C^\infty$
on $\dC$. Then, the implicit function theorem implies that $z \mapsto f(z)$ is
$C^\infty$ on $\Si$.

Now, take $z \notin \Si$, we recall that 
\[
h(t,z) = \dE \frac{h(t,z)+t}{|G_z|^2+(h+t)^2},
\]
and for $0 < t \leq 1$, $0 \leq h(z,t) \leq c$. Hence, letting $t \downarrow
0$, any accumulation point $f(z)$ of $h(z,t)$ satisfies $f(z) \in [0 ,c]$ and
\[
f(z) = \dE \frac{f(z)}{|G_z|^2+f(z)^2}.
\]
If $f(z) \ne 0$ then \eqref{eq:FPfz} would hold true. However, this would 
contradict the assumption $z \notin \Si$. Therefore, for all $z \notin
\Si$, we have
\[
f(z) = 0. 
\]
By \eqref{eq:fixpointgamma}, it follows that 
\[
\beta (z) :=  \lim_{t \downarrow 0}  \beta ( q ( z,it) ) %
= -\dE\frac{G_z }{|G_z|^2+f(z)^2}
\]
By lemma \ref{le:propRes}, the Brown measure of $c+g$ is equal in distribution
to
\[
\mu_{c + g}  = - \frac{1}{ \pi } \pd \beta (z) %
= \frac{1}{ \pi} \pd \dE \frac{G_z}{|G_z|^2+f(z)^2}.
\]
Now, if $z \notin \bar \Si$, then $f(z)$ is $0$ in a neighborhood of $z$.
Hence, $ - \beta(z) = \dE (\bar G - \bar z)^{-1}$. Since $\pd \bar z=0$, $\pd
\beta(z) = 0$, and we deduce that the density of $\mu_{c+g}$ is $0$ on $(\bar
\Si)^c$.

Assume now that $z \in \Si$. We find that $\mu_{c + g} $ has a density given by $1/\pi$ times 
\[
- \pd \beta (z) 
=   \dE \frac{1}{|G_z|^2+f(z)^2}  %
- \dE \frac{|G_z|^2}{(|G_z|^2+f(z)^2)^2} %
-  2 f(z)f'(z)\dE\frac{G_z}{(|G_z|^2+f(z)^2)^2},
\]
where we use $\pd G_z=1$, $\pd |G_z|^2=\bar G_z$. Here $f'(z)=\pd f(z)$. Using
\eqref{eq:FPfz}, the first term on the right hand side is equal to $1$ and
\[
0 =  \dE \frac{\bar G_z}{(|G_z|^2+f(z)^2)^2} %
+2f(z)f'(z)\dE\frac{1}{(|G_z|^2+f(z)^2)^2}. 
\]
Hence,
\begin{align*}
  - \pd \beta (z)
  & =  1 - \dE \frac{|G_z|^2}{(|G_z|^2+f(z)^2)^2} %
  +\frac{\ABS{\dE\frac{G_z}{(|G_z|^2+f(z)^2)^2}}^2}%
  {\dE\frac{1}{(|G_z|^2+f(z)^2)^2}} \\
  & = \dE\frac{f(z)^2}{(|G_z|^2+f(z)^2)^2} %
  +\frac{\ABS{\dE\frac{G_z}{(|G_z|^2+f(z)^2)^2}}^2}%
  {\dE\frac{1}{(|G_z|^2+f(z)^2)^2}}.
\end{align*}
From what precedes, if $ z \in \Si$, $f(z) > 0$. We have thus proved that the
density of $\mu_{c + g} $ is positive on $\Si$, the interior of $\bar \Si$,
and given by $1/\pi$ times the above expression, while on $\dC \backslash \bar
\Si$ the density is $0$. In particular,  the support of $\mu_{c + g} $ is
$\bar \Si$. This concludes the proof of theorem \ref{th:propmu}.

\section{Extremal eigenvalues: Proof of theorems \ref{th:support} and \ref{th:support2}}
\label{se:proof:th:support}

\subsection{Proof of theorem \ref{th:support}}
The next lemma allows us to control the spectral norm of the diagonal matrix $\Ul D$ defined in \eqref{bars}. The proof uses
a refined central limit theorem together with estimates for the maximum of
i.i.d.\ standard Gaussian random variables. We refer to \cite[theorem
1.5]{MR2206341} for a proof.

\begin{lemma}\label{le:maxDii}
  Under the assumptions of theorem \ref{th:support}, almost surely
  \[
  \max_{1\leq i\leq n} |\Ul D_{ii}| = \si \sqrt{2n\log(n)}\,(1+o(1)).
  \]
\end{lemma}

Next, we observe that from the Bauer-Fike theorem \cite[theorem
25.1]{MR2325304} one has that the eigenvalues of $\Ul L=\Ul X-\Ul D$ are all
contained in the subset of $\dC$ defined by
\[
\bigcup_{i=1}^n B(-\Ul D_{ii},s_1(\Ul X)),
\]
where $B(z,t)$ stands for the Euclidean closed ball (actually a disk) around
$z$ with radius $t$, and $s_1(\Ul X)$ is the largest singular value of $\Ul
X$. Since $s_1(\Ul X)=2\si \sqrt{n}\,(1+o(1))$ (by \cite[theorem
2]{MR1235416}), using lemma \ref{le:maxDii} one finds that all eigenvalues
$\Ul \la$ of $\Ul L$ must satisfy \eqref{radius1}.

We turn to the proof of \eqref{radius2}. We observe that from the Bauer-Fike
theorem all eigenvalues of $L$ must be contained in the subset of $\dC$
defined by
\[
\bigcup_{i=1}^n B(\zeta_i,s_1(\Ul X)),
\]
where $\zeta_1 \geq \dots \geq \zeta_n$ are the ordered (real) eigenvalues of $-\Ul D+mJ-mnI$.
The eigenvalues of $mJ-mnI$ are easily seen to be $z_1=0$ and
$z_2=\cdots=z_n=-mn$.   Thus, the bound $s_1(\Ul
X)=\si \sqrt{n} ( 2 + o(1))$ proves statement  \eqref{radius2} on $\Im (\la)$. 

For the statement on $\Re (\la)$, we first notice that under our assumptions one certainly has 
\begin{equation}\label{eq:hyp_sig3}
  \varlimsup_{n \to \infty} \frac{ 2 \si }{m} \sqrt{ \frac{2  \log(n) }{  n }} < 1.
\end{equation}
Moreover, from Weyl's inequality, $\max_{1\leq j\leq n}|\zeta_j -z_j| \leq \max_i|\Ul D_{ii}|$, so that one has
\[
|\zeta_1| \leq \max_{1\leq i\leq n}|\Ul D_{ii}| %
\quad\text{and}\quad %
\max_{2\leq j\leq n}|\zeta_j+mn| \leq \max_{1\leq i\leq n}|\Ul D_{ii}|.
\]
If $n$ is large enough, by lemma \ref{le:maxDii}, the bound $s_1(\Ul
X)=\si \sqrt{n} ( 2 + o(1))$, and using \eqref{eq:hyp_sig3}, we see that 
\[
B(\zeta_1,s_1(\Ul X))\cap B(\zeta_j,s_1(\Ul X))=\varnothing,
\]
for all $2\leq j\leq n$. Thus, a continuity argument \cite[proof of
Gershgorin's theorem 6.1.1]{MR1084815} implies that apart from the trivial eigenvalue $\la=0$,
which belongs to $B(\zeta_1,s_1(\Ul X))$, all other eigenvalues of $L$ belong
to $\cup_{j=2}^n B(\zeta_j,s_1(\Ul X))$. Using the bound $s_1(\Ul X)=2\si
\sqrt{n}\,(1+o(1))$ and lemma \ref{le:maxDii} we see that any $\la\neq 0$
in the spectrum of $L$ must satisfy \eqref{radius2}.

\subsection{Proof of theorem \ref{th:support2}}

Since $\| \Ul L \| \leq \| \Ul X \| + \| \Ul D \| $, we may bound separately
the norms of $\| \Ul X \|$ and $\| \Ul D \|$. We have
\[ 
\| \Ul D \| =  \max_{1 \leq i \leq n} \ABS{ \sum_{j =1 } ^ n \Ul X_{ij} }. 
\]
By assumption there exists $a > 0$ such that with probability one, $| \Ul
X_{ij} | \leq a$. From Bennett's inequality, for any $1 \leq i \leq n$, $t>0$:
\[
\dP \PAR{ \ABS{ \sum_{j =1 } ^ n \Ul X_{ij} } \geq t \si \sqrt{n}} \leq 2 \exp \PAR{  - \frac{ \si^2 n } { a^2 } h \PAR{ \frac{ a t}{ \si \sqrt{n}} } },  
\]
where $h(s) = (1 +s ) \log (1 + s ) - s \sim s^2 / 2 $ as $s$ goes to $0$. We
choose $t = \sqrt{2c\log(n)}$, we find by \eqref{eq:hyp_sig4}
\[
\dP \PAR{ \ABS{ \sum_{j =1 } ^ n \Ul X_{ij} } %
  \geq \si \sqrt{2cn\log(n)}} %
\leq n^ { - c ( 1 + o(1)) }.
\]
In particular, if $c  =  2 (1+ \veps)^2$, from the union bound, for $n \gg1 $, 
\[
\dP\PAR{\max_{1\leq i\leq n}\ABS{\sum_{j=1}^n\Ul X_{ij}} %
  \geq 2 ( 1 + \veps) \si \sqrt{n\log(n)}} %
\leq n^ { - 1 + \veps }.
\]
Hence, from Borel-Cantelli lemma we get  a.s. for $n \gg 1$, 
\[
\| \Ul D \| \leq   ( 2 + o(1))  \si \sqrt{n\log(n)}. 
\]
We now turn to the bound on $\| \Ul X\|$. This is a much more delicate matter.
Fortunately, we may use a result by Vu \cite[theorem 1.4]{MR2384414}, which
extends F\H{u}redi and Koml\'os \cite{MR637828}. It asserts that a.s. for $n
\gg 1$,
\[
\NRM{\Ul X} = (2+o(1)) \si\sqrt{n} +c\si^{\frac{1}{2}}n^{\frac{1}{4}}\log(n).
\]
This proves \eqref{radius12}. 
To prove \eqref{radius22}, observe that by assumption one has again \eqref{eq:hyp_sig3}.
Thus, we may repeat the argument in the
proof of theorem \ref{th:support}.

\section{Invariant measure: Proof of theorem \ref{th:invmeas}}

\label{sec:invmeas}

We start by proving the matrix $L$ is irreducible. Consider the graph $G$ on
$\{ 1, \ldots , n\}$ whose adjacency matrix is $A = (A_{ij} )_{1 \leq i,j \leq
  n}$ with $A_{ii} = 0$ and $A_{ij} = \IND_{\{X_{ij} \ne 0\}}$ for $i \ne j$.
Note that $G$ is an oriented Erd\H{o}s-R\'enyi random graph where each edge is
present independently with probability $p := \dP(X_{ij} \ne 0)$. To prove 
irreducibility of $L$, we shall prove that the oriented graph $G$ is
connected. Hence, by a fundamental result of Erd\H{o}s and R\'enyi
\cite{MR0120167}, see also e.g.\ \cite[theorem 7.3]{MR1864966},
\cite[Corollary 3.31]{MR1782847}, it is sufficient to check that
\begin{equation}
  \label{eq:limp}
  \varliminf_{n \to \infty} \frac{ p  n } {\log(n) }  > 1. 
\end{equation}
Note that the cited references deal with non-oriented Erd\H{o}s-R\'enyi
graphs. This does not change much, see the discussion \cite[p.~2]{MR2885424}.
For either model A or model B, the statement \eqref{eq:limp} follows from
assumption \eqref{eq:hyp_sig2}.

We may now turn to the analysis of the invariant measure $\Pi$. We will rely
on a method that has already been successfully used in random matrix models
with finite rank perturbations, for example in \cite{MR2782201,MR3010398}.
We write
\[
L = \Ul L + m \phi \phi ^ * - m n I,
\]
with $\phi = (1, \ldots, 1)^\top$. Then, if $z$ is not an eigenvalue of $\Ul
L$, we have
\begin{align*}
  \det ( L - z I + m n I ) %
  & = \det(\Ul L + m \phi \phi ^ * - z I) \\
  & = \det(I+  m ( \Ul L  - z I)^{-1} \phi\phi^*)\det(\Ul L- z I) \\
  & = \PAR{1 +m\phi ^ * (\Ul L-zI)^{-1}\phi}\det ( \Ul L  - z I) ,
\end{align*}
where at the last line we have used the well known Sylvester determinant theorem: for all
$A \in \cM_{n,p} (\dC)$, $B \in \cM_{p,n} (\dC)$,
\[
\det ( I_n + A B ) = \det ( I_p + B A  ).
\]
In particular, $z - m n$ will be an eigenvalue of $L$ if and only if
\[
f(z) := 1 +  m \phi ^ * ( \Ul L  - z I)^{-1} \phi  = 0.
\]
Further, 
if $z$ is not an eigenvalue of $\Ul L$,
\begin{align*}
  \phi^* ( \Ul L  - z I)^{-1}   ( L - z I  + mn I   ) %
  & = \phi^* ( \Ul L  - zI )^{-1}   ( \Ul L - z I  + m \phi \phi ^ *   ) \\
  & = f(z) \phi^*.
\end{align*}
We deduce that if $z - m n$ is an eigenvalue of $L$ then $\phi^*(\Ul
L-zI)^{-1}$ is a left eigenvector of $L$ with eigenvalue $z-mn$. Now $0$ is an
eigenvalue of $L$ and, by theorems \ref{th:support}-\ref{th:support2}, a.s.\
for $n\gg 1$, $mn$ is not an eigenvalue of $\Ul L$. Hence from what precedes,
for $1 \leq i \leq n$,
\[
\Pi_i = \frac{ u_i} { \sum_{k=1} ^ n u_k},
\]
where 
\[
u^* =   \phi^ * \PAR{  I - \frac{\Ul L}{mn}}^{-1}.
\]
From the resolvent identity, 
\[ 
\PAR{  I - \frac{\Ul L}{mn}}^{-1} - I %
= \frac{\Ul L}{mn} \PAR{  I - \frac{\Ul L}{mn} }^{-1}  .
 \]
Thus, using theorems \ref{th:support}-\ref{th:support2}, and $\|\phi\|=\sqrt n$,
\[
\NRM{\phi - u} %
\leq \NRM{\phi}\NRM{\frac{\Ul L}{mn}}\NRM{\PAR{ I - \frac{\Ul L}{mn} }^{-1}} %
= \cO\PAR{ \frac{\si} { m} \sqrt{\log(n)}} %
+\cO\PAR{ \frac{\sqrt{\si}} { m} \frac{ \log(n) }{n^{1/4}} } . %
\]
In other words, writing $u_i = 1 + \veps_i$, we find
\begin{equation*}
\sum_{k=1}^n | \veps_k | %
\leq \sqrt{n}\PAR{ \sum_{k=1}^n \veps^2_i }^{\frac{1}{2}} %
= \cO\PAR{ \frac{\si} { m} \sqrt{n\log(n)}} %
+\cO\PAR{ \frac{\sqrt{\si}} { m} n^{1/4} \log(n) } . %
\end{equation*}
Consequently, uniformly in $i=1,\dots,n$:
\[
\Pi_i = \frac{ 1 + \veps_i } { n +  \sum_{k=1} ^ n \veps_k} %
= \frac{1}{n} + \frac{ \veps_i}{n} %
+  \cO\PAR{\frac{\si}{m}\sqrt{\frac{\log(n)}{n^3}}} %
+ \cO\PAR{\frac{\sqrt{\si}}{m}\frac{\log(n)}{n^{7/4}}}.
\]
Therefore,
\begin{align*}
  \sum_{i=1}^n \ABS{\Pi _i - \frac{1}{n}} 
  & \leq  \frac{1}{n} \sum_{i=1}^n\PAR{\ABS{\veps_i} %
    +\cO\PAR{\frac{\si}{m}\sqrt{\frac{\log(n)}{n}}} %
    +\cO\PAR{\frac{\sqrt\si}{m}\frac{\log(n)}{n^{7/4}}}}\\
  & = \cO\PAR{\frac{\si}{m}\sqrt{\frac{\log(n)}{n}}} %
  +\cO\PAR{\frac{\sqrt\si}{m}\frac{\log(n)}{n^{3/4}}}.
\end{align*}


        

\appendix

\section{Concentration function and small ball probabilities}

For $x \in \dC^n$ and $ t \geq 0$, define the concentration function  as
\begin{equation}\label{concfunc}
  p_x ( t ) %
  = \max_{w \in \dC} \dP \PAR{ \ABS{ \sum_{i=1} ^ n X_{i} x_ i - w  } \leq t  } ,  
\end{equation}
where $\{X_i, \,i=1,\dots,n\}$ are i.i.d.\  copies of the complex valued random
variable $\bx$ with law $\cL(n)$. Throughout this section we assume that $\bx$
satisfies \eqref{eq:hyp_tightness1} and $n\si^2\geq 1$ only. This implies the
existence of constants $a,b > 0$ such that for all $n \geq 1$, letting
$\bx^{(a)}$ denote the random variable $\bx$ conditioned on $|\bx|\leq a$:
\begin{gather}\label{eq:defstilde}
  \dP \PAR{ | \bx|  \leq a } \geq b\,,\quad %
  \Wt \si^2 =  \VAR\PAR{\bx^{(a)}} \geq b \si^2,
  \\
  \label{eq:defstilde2}
  \dP ( a^ {-1} \leq | \bx  |  \leq a ) \geq b \si^2.
\end{gather}

We start with a  consequence of Kolmogorov-Rogozin inequality \cite{MR0101545,MR0131894}. 

 \begin{theorem}[Concentration bound]\label{th:halasz}
   Let $ 1 \leq m \leq n$. There exists a constant $C>0$ independent of
   $(m,n)$ such that if $|x_i | \geq 1$, $1 \leq i \leq m $, then for all $t
   \geq 0$,
   \[
   p_x \PAR{ t } \leq  \frac{ C }{\si \sqrt{m}} \PAR{ t + 1 } . 
   \]
 \end{theorem} 

\begin{proof}
  For any constant $c > 0$, $e^{ -c m}$ is smaller than
  $\frac{1}{\si\sqrt{m}}$ for all $m \gg 1$. Hence, from Hoeffding inequality,
  it is sufficient to prove the statement conditioned on $\{ | X_{1} | \leq a,
  \ldots, |X_{k} | \leq a \}$ with $k = b m /2$, if $b>0$ is as in
  \eqref{eq:defstilde}. At the price of replacing $m$ by $bm/2$ and $\si$ by
  $\Wt\si$, one can replace $\bx$ by $\bx^{(a)}$ from the start. For ease of
  notation, we will simply assume that the variables are bounded by $a$.

  We first assume that all $X_{i}$ and $x_{i}$ are real valued. It is
  sufficient to prove the statement for $t \geq 1$. Following Rudelson and
  Vershynin in \cite{MR2407948}, set $Y = | X_{1} - X_{2} | /2$. For any $t >
  0$,
  \[
  \dP  ( |Y| \geq t ) \geq \dP ( |X_1 |\geq 3 t  ) \dP ( |X_2 |\leq t ). 
  \]
  Using \eqref{eq:defstilde2}, \eqref{eq:hyp_tightness1} and Markov
  inequality, we find, for some constant $c_0 > 0$,
  \[
  \dP  ( |Y| \geq 1 / ( 3a)  ) \geq b \si^2 ( 1 - \dE |X_2| ^2 a^2 ) \geq b \si^2 ( 1 - c_0 \si^2 a^2 ). 
  \]
  In particular, if $c_0 \si^2a^2 < 1/2 $ then there exists a constant $c > 0$
  such that
  \[
  \dP ( |Y| \geq c ) \geq c\si^2. 
  \]
  On the other hand, if $c_0 \si^2a^2 \geq 1/2 $, from \eqref{eq:defstilde2},
  we may use Paley-Zygmund inequality and deduce easily that the above
  inequality also holds (for some new constant $c$).

  Let $\hat Y$ the law of the variable $Y$ conditioned on $|Y|\geq c$. It
  follows that if $f$ is a non-negative measurable function,
  \begin{equation}\label{eq:fYhatY}
    \dE f (Y ) \geq c \si^2 \dE f ( \hat Y ). 
  \end{equation}
  From Esseen inequality, cf.\ \cite[lemma 4.2]{MR2407948}, for some universal
  constant $C> 0$,
  \[
  p_x ( t) %
  \leq C \int_{0}^{2\pi}\!\ABS{\dE e^{i\frac{\te}t\sum_{k=1}^m X_{k}x_k}}\,d\te %
  = C\, t \int_{0}^{2\pi t^{-1}}\!\prod_{k=1}^m\ABS{\dE e^{i\te X_{k}x_k}}\,d\te.
  \]
  Now, we use the bound $|z| \leq \exp(- \frac{1}{2}(1-|z|^2))$ valid for all
  $z$, and the identity,
  \[
  \ABS{ \dE e^ { i \te X_{k}x_k  } }^2 =  \dE \cos ( 2 \te Y x_k ),
  \]
  to write
  \[
  \prod_{k=1}^n \ABS{ \dE e^ { i \te X_{k} x_k } }\leq \prod_{k=1}^m 
  \exp{(-\frac12(1-\dE\cos ( 2 \te Y x_k))}=\prod_{k=1}^m
  \exp{(-\dE\sin^2 ( \te Y x_k))}.
  \]
  Therefore,
  \[
  p_x ( t) \leq C\, t \int_{ 0}^ {2\pi t^{-1}}  \exp { \Big(-  \sum_{k=1} ^ m \dE  \sin ^ 2 ( \te x_k Y ) \Big)} d\te. 
  \]
  This implies, using \eqref{eq:fYhatY} and $|x_k|\geq 1$, 
  \[
  p_x ( t)
  \leq C\,t \,\sup_{\la \geq c} \int_{0}^ {2\pi t^{-1}}  \exp {\Big( - c\, m \si^2 \sin ^ 2 ( \te \la ) \Big)} d\te.  
  \]
  By the change of variable $\te'=\te\la$, using $t^{-1} \leq 1$, one
  has 
  \begin{equation}\label{pxtla}
    p_x ( t)  \leq  C\, t   \,\sup_{\la \geq c } \frac{ 1} {\la}  \int_{0}^ { 2\pi\la}  e^ { - c\, m \si^2 \sin ^ 2 ( \te ) } d\te 
    \leq \frac{ C'\,t }{ \si \sqrt{m}}.  
  \end{equation}
  The right hand side follows easily by decomposing the integral along the
  periods of $\sin(\te)$.

  It remains to prove the statement for complex random variables. It is a
  consequence of the real case. Indeed, we may assume for example that at
  least $m/2$ of the $x_i$'s satisfy $\Re ( x_i ) \geq 1/ \sqrt{2}$
  (otherwise, this is satisfied by the imaginary part). We then notice that
  the function $p_x$ does not change if we rotate $\bx$ into $e^{i u} \bx$
  with $u \in [0,2\pi]$. But, we may argue as in \cite[lemma 2.4]{MR2409368} :
  there exists $u$ such that for all $z\in \dC$,
  \[
  \VAR \Big( \Re(  e^{i u}  \bx  z )\Big)  \geq \Re(z) ^ 2 \frac{ \si^2}{2}.  
  \] 
  Finally, we note that 
  \[
  p_x ( t ) \leq \max_{w \in \dC} \dP \Big( \Big|\Re \Big( \sum_{k=1} ^ m e^{i
    u} X_{k} x_ k - w \Big) \Big| \leq t \Big) = \max_{w \in \dR} \dP \Big(
  \Big| \sum_{k=1} ^ m \Re( e^{i u} X_{k} x_ k) - w \Big| \leq t \Big) .
  \]
  Therefore, the case of complex random variables follows from the case of
  real random variable. 
\end{proof}

If $\si\sqrt{m}$ is not large, the bound in theorem
\ref{th:halasz} becomes useless. In this case, we may however give a weak
bound on the concentration function.

\begin{lemma}\label{le:GT}
  Let $ 1 \leq m \leq n$. There exists a constant $c_0 >0$ independent of
  $(m,n)$ such that if $x$ satisfies $|x_i|\geq 1$, for $1\leq i \leq m$, then
  \begin{equation}\label{claimGT}
    p_x \PAR{c_0} \leq 1- c_0   ( m \si^2  \wedge 1) \,.  
  \end{equation}
\end{lemma} 

\begin{proof}
  Let us start with some comments. First from \eqref{eq:defstilde2},
  \eqref{eq:hyp_tightness1} and Markov inequality,
  \[
  \dP  ( |X_1 | \geq 1 / ( 2a)  ) \leq  4 \dE |X_1| ^2 a^2  \leq  C \si^2 a^2 \quad  \hbox{ and } \quad  \dP  ( |X_1  | \geq 1 / a  ) \geq b\si^2. 
  \]
  Hence, if $|x_1| \geq 1$, for $n = 1$, we find $p_{x_1} \PAR{ 1 / ( 2a ) }
  \leq 1- c_0 ( \si ^2 \wedge 1 )$. Since the concentration function of the
  sum is bounded by the concentration function of one of its summands, we get
  \eqref{claimGT} for $m=1$.

  Also, from theorem \ref{th:halasz}, if $ \si \sqrt{m}\geq 4 C$, then
  $p_x \PAR{1} \leq 1/2$. Hence, it suffices to prove the statement for $
  \si^2 m \leq 16 C^2$. Then, simple manipulations show that, at the price of reducing the constant 
%
$c_0$, it suffices to prove the statement of
  the lemma with the extra assumptions $ \si^2 m \leq \veps$ and $m \geq 1 /
  \veps$ for some $\veps > 0$ arbitrarily small (but independent of $m,n$).

  We will first give a proof in the specific case when $\bx$ is a Bernoulli
  random variable with parameter $p$. We will generalize the argument
  afterward. We fix $t = 1/4$ and define $p_k = p_{(x_1, \ldots ,x_k)} (t)$.
  We will assume that $p m \leq 1/4$ and $m \geq 4$. We are going to show that
  $p_m \leq 1 - m p / 4$. Let $k \geq 2$ and assume that $p_{k-1} \geq 1 - kp
  $. Let $S = \sum_{i =1} ^{k-1} X_i x_i $ and $u_*$ be such that $p_{k-1} =
  \dP \PAR{ |S - u_*| \leq t }$. We write, for $u \in \dC$,
  \[
  \dP \PAR{ |S+ x_k X_k  - u  | \leq t } = ( 1 - p )  \dP \PAR{ |S - u  | \leq t }  + p \dP \PAR{ |S - u + x_k   | \leq t }  
  \]
  Now, since $|x_k |\geq 1$, $u$ and $u - x_k$ cannot be both at distance less
  than $t = 1/4$ from $u_*$. Using $\dP ( |S - u_* | > t ) \leq k p$ and $\dP
  ( |S - u | \leq t ) \leq p_{k-1}$, we deduce
  \begin{align*}
    p_k & \leq \max \PAR{ ( 1 - p ) p_{k-1}  +  k p^2    ,  ( 1 - p ) k p  + p p_{k-1}   }  \\
    & \leq   \max \PAR{   p_{k-1}  - p (  1 - 2 k p  )  ,   p ( k + 1 - 2 k p  )   } \\
    & \leq  p_{k-1}  - p/ 2,
  \end{align*}
  provided that $pk \leq 1/4$. Now the argument goes as follows, if there
  exists $\lfloor m / 2 \rfloor \leq k < m $ such that $p_{k} \leq 1 - k p $,
  then since $p_{m} \leq p_{k}$, we find $p_m \leq 1 - \lfloor m / 2 \rfloor p
  \leq 1 - m p / 4$. Otherwise, for all $\lfloor m / 2 \rfloor \leq k < m$,
  $p_k \geq 1 - k p$, and we apply recursively the above argument. We find
  $p_m \leq p_{\lfloor m / 2 \rfloor} - ( m -\lfloor m / 2 \rfloor) p / 2 \leq
  1 - m p / 4$. This concludes the proof when $\bx$ is Bernoulli
  random variable.

  In the general case, we take $ t = 1/ ( 4 a)$, with $a>0$ as in \eqref{eq:defstilde}. 
  We have $ p = b \si^2 \leq
  \dP ( |X_1 | \geq 1 / a )$, and, from Markov inequality, $1 - q = 1 - c
  \si^2 \leq \dP ( |X_1 | \leq 1 / (4a) )$. Setting $p_k = p_{(x_1, \ldots
    ,x_k)} (t)$, the above argument gives $p_m \leq 1 - m p / 4$ as soon as
  $\si^2 m$ is small enough. Indeed, as above, let $k \geq 2$ and assume that
  $p_{k-1} \geq 1 - k p $. Note that, with the above notation, if $|X_k - X'_k
  | \geq 3 /( 4 a)$, then $u + x_k X'_k$ and $u + x_k X_k $ cannot be both at
  distance less than $t = a/4$ from $u_*$. In particular, from
  \begin{align*}
    \dP \PAR{ |S+ x_k X_k  - u  | \leq t } 
    \ \leq\ & ( q -p )  p_{k - 1} \\
    &+  ( 1 - q )  \dP \PAR{ |S - u  + x_k X_k | \leq t |\;  |X_k|\leq 1 / (4a)  }  \\ 
    & + p \dP \PAR{ |S - u + x_k  X_k  | \leq t  |\; |X_k|\geq 1 / a   },   
  \end{align*}
  we get
  \begin{align*}
    p_k 
    &\leq\max\PAR{(q-p)p_{k-1}+(1-q)p_{k-1}+kp^2,(q-p)p_{k-1}+(1-q)kp+pp_{k-1}}\\
    &\leq\max\PAR{p_{k-1}-p(1-2kp),p(k+\frac{q}{p}-2kq)}\\
    &\leq p_{k-1}-p/2,
  \end{align*}
  provided that $\si^2 k$ is small enough. The rest of the argument is
  identical.
 \end{proof}

 The following corollary is a version of \cite[lemma
 4.6]{gotze-tikhomirov-new} (our variable $\bx$ satisfies however more general
 statistical assumptions than in \cite{gotze-tikhomirov-new}).
 \begin{corollary}\label{cor2:halasz}
   There exists $c > 0$ such that if $x \in \INCOMP ( \delta, \rho)$, with
   $\de,\rho\in(0,1)$, $u \in \dC^n$, then
   \begin{equation}\label{Xxu}
     \dP ( \| X x - u \|_2 \leq c    \, \rho   \sqrt{\de}) %
     \leq e^ { - c\, n ( 1 \wedge \si^2  \delta n )  }.
   \end{equation}
 \end{corollary} 
 \begin{proof}
   We have \[ \| X x - u \|_2 = \sqrt{ \sum_{i=1} ^ n S_i ^ 2 } \quad \hbox{
     with } \quad S_i = \ABS{ \sum_{k=1} ^ n X_{ik} x_k - u_i }.\] Thanks to
   lemma \ref{le:incompspread} one may assume that $|x_i|\geq \rho/\sqrt{n}$
   for all $1\leq i\leq \de n/2$. From theorem \ref{th:halasz}, one obtains
   \[
   \dP \PAR{ S_i \leq \frac{ \rho t}{\sqrt{n}} } %
   \leq p_x(\rho t/\sqrt{n}) %
   \leq \frac{ C\, (t+1)}{ \si \sqrt{\delta n } },
   \]
   for some constant $C>0$, and all $t\geq 0$. Assume first that $\si
   \sqrt{\delta n} \geq 4 C$. Then, taking $t = 1$, we get \[ \dP \PAR{ S_i
     \leq \frac{ \rho }{ \sqrt{n}} } \leq \frac{ 1}{ 2 }.
   \]
   The $S_i$'s are independent random variables. By Hoeffding's inequality,
   \[
   \dP \PAR{ \| X x - u \|_2 \leq \rho/2 } %
   \leq \dP\PAR{\sum_{i = 1}^n\IND_{\{S_i\geq\frac{\rho}{\sqrt{n}}\}}\leq n/4} %
   \leq e^{ - n/8 }.
   \]
   This implies \eqref{Xxu}.
   
   It remains to consider the case $\si \sqrt{\delta n} <4 C $. Here, we may
   use lemma \ref{le:GT} with $m=\de n/2$. Therefore, there exists $c_0 > 0$,
   such that
   \begin{equation}\label{claimGT1}
     \dP \PAR{ S_i \leq \frac{ \rho c_0 }{\sqrt{n}} } \leq  p_x \PAR{ \frac{ \rho c_0 }{ \sqrt{n}} } \leq 1- c_0  \si^2  \delta n  .  
   \end{equation}
   Once \eqref{claimGT1} is available, we can conclude the proof as follows.
   From Bennett's inequality, there exists $c>0$ such that if $\varepsilon_i$
   are i.i.d.\  Bernoulli random variables with parameter $q$, $0 < q < 1/2$,
   \[
   \dP \PAR{  \sum_{i=1} ^ n \veps_i \leq  \frac{ n q } { 4} }  \leq \dP \PAR{  \sum_{i=1} ^ n \dE \veps_i - \veps_i   \geq  \frac{ 3 n q } {  4} } \leq e^ { - c n q } . 
   \]
   From \eqref{claimGT1}, we can take $\varepsilon_i=
   \IND_{\{S_i>\frac{c_0\rho}{\sqrt{n}}\}}$, with $1/2\geq q\geq
   c_0n\si^2\de$. Then, using $n\si^2\geq 1$, one has that $\| X x - u \|_2
   \leq c\rho\sqrt\de$ implies $\sum_{i=1} ^ n \veps_i \leq \frac{ n q } {
     4}$, for some $c>0$. It follows that for some new $c>0$:
   \[
   \dP \PAR { \| X x - u \|_2 \leq  c       \rho \sqrt{  \de  } }   \leq e^ { - c n^2 \si^2  \delta }. 
   \]
   This ends the proof.
\end{proof} 

Corollary \ref{cor2:halasz} implies a probabilistic bound on the distance of
$Xx$ to a vector space.
\begin{corollary}\label{cor3:halasz}
  There exists $c_1 > 0$ such that if $x \in \INCOMP ( \delta, \rho)$, with
  $\de,\rho\in(0,1)$, $s\geq 1$, and $H$ is a deterministic vector space such
  that
  \begin{equation}\label{dimh}
    \dim(H) \leq   
    \frac{  c_1\, n ( 1 \wedge \si^2  \delta n ) } { \log ( \frac{ s} {\rho^2 \delta } )},
  \end{equation} 
  then
  \[
  \dP \PAR{  \DIST( X x, H )  \leq  c_1 \, \rho   \sqrt{\de}  \, ; \, s_1 ( X ) \leq s  } \leq e^ { - c_1\, n ( 1 \wedge \si^2  \delta n )  }.
  \]
\end{corollary} 
\begin{proof}
  By definition, 
  \[
  \DIST( X x, H ) = \min_{ u \in H } \| X x - u  \|_2. 
  \]
  On the event $s_1(X) \leq s$, if $\| u \|_2 \geq 2s $ then $\| X x - u \|_2
  \geq s \geq 1\geq \rho \sqrt{\de}$. Hence, we may restrict our attention to
  \[
  \min_{ u \in B_H(s)  } \| X x - u  \|_2, 
  \]
  where $B_H (s) = \{ u \in H ; \| u \|_2 \leq 2 s\} $. Now, if $\veps < 1$
  and $k = \dim ( H)$, there exists an $\veps$-net $N$ of $B_H (s)$ of
  cardinality $( 5 s / \veps )^{ k}$,
  with
  \[
  \min_{ u \in B_H(s)  } \| X x - u  \|_2 \geq  \min_{ u \in N } \| X x - u  \|_2 - \veps.
  \]
  Let $c>0$ and take $\veps = \frac{c \rho \sqrt{\de}}{2}$. From the union
  bound,
  \[
  \dP \PAR{ \DIST( X x, H ) %
    \leq \frac{c \rho \sqrt{\de}}{2} \, ; \, s_1 ( X ) \leq s } %
  \leq \PAR{ \frac{ 10 s }{ c \rho \sqrt{\delta}} }^{ k} \max_{u \in \dC^n}
  \dP \PAR{ \| X x - u \|_2 \leq c \rho \sqrt{\de } }.
  \]
  Then if $c$ is chosen as in corollary \ref{cor2:halasz}, we find
  \[
  \dP\PAR{\DIST(Xx,H)\leq\frac{c\rho\sqrt{\de} }{2}\,;\,s_1(X)\leq s} %
  \leq e^{k\log\PAR{\frac{10s}{c\rho\sqrt{\delta}}} -cn(1\wedge\si^2\delta n)}.
  \]
  In particular if $k=\dim(H)$ satisfies \eqref{dimh}, then the conclusion
  follows by taking $c_1$ sufficiently small.
\end{proof}




\section*{Acknowledgments}
The idea of studying this model emerged from animated discussions with Laurent
Miclo in Marne-la-Vall\'ee. We are grateful to Philippe Biane for helpful
suggestions regarding the algebraic literature. This work was supported by the
French ANR 2011 BS01 007 01 GeMeCoD, the GDRE GREFI-MEFI CNRS-INdAM, and the
European Research Council through the ``Advanced Grant'' PTRELSS 228032.


\begin{thebibliography}{10}

\bibitem{adamczak-chafai}
R.~Adamczak and D.~Chafa{\"{\i}}, \emph{Circular law for random matrices with
  unconditional log-concave distribution}, preprint
  \url{http://arxiv.org/abs/1303.5838}, 2013.

\bibitem{AGZ}
Greg~W. Anderson, Alice Guionnet, and Ofer Zeitouni, \emph{An introduction to
  random matrices}, Cambridge Studies in Advanced Mathematics, vol. 118,
  Cambridge University Press, Cambridge, 2010.

\bibitem{MR1062321}
A.~L. Andrew, \emph{Eigenvalues and singular values of certain random
  matrices}, J. Comput. Appl. Math. \textbf{30} (1990), no.~2, 165--171.

\bibitem{MR1235416}
Z.~D. Bai and Y.~Q. Yin, \emph{Limit of the smallest eigenvalue of a
  large-dimensional sample covariance matrix}, Ann. Probab. \textbf{21} (1993),
  no.~3, 1275--1294.

\bibitem{MR2782201}
Florent Benaych-Georges and Raj~Rao Nadakuditi, \emph{The eigenvalues and
  eigenvectors of finite, low rank perturbations of large random matrices},
  Adv. Math. \textbf{227} (2011), no.~1, 494--521.

\bibitem{MR2325304}
Rajendra Bhatia, \emph{Perturbation bounds for matrix eigenvalues}, Classics in
  Applied Mathematics, vol.~53, Society for Industrial and Applied Mathematics
  (SIAM), Philadelphia, PA, 2007, Reprint of the 1987 original.

\bibitem{MR1488333}
Ph. Biane, \emph{On the free convolution with a semi-circular distribution},
  Indiana Univ. Math. J. \textbf{46} (1997), no.~3, 705--718.

\bibitem{MR1876844}
Ph. Biane and F.~Lehner, \emph{Computation of some examples of {B}rown's
  spectral measure in free probability}, Colloq. Math. \textbf{90} (2001),
  no.~2, 181--211.

\bibitem{MR1864966}
B.~Bollob{\'a}s, \emph{Random graphs}, second ed., Cambridge Studies in
  Advanced Mathematics, vol.~73, Cambridge University Press, Cambridge, 2001.
  \MR{MR1864966 (2002j:05132)}

\bibitem{bordenave-caputo-chafai-ii}
Charles Bordenave, Pietro Caputo, and Djalil Chafa{\"{\i}}, \emph{Spectrum of
  large random reversible {M}arkov chains: {H}eavy-tailed weigths on the
  complete graph}, {A}nnals of {P}robability \textbf{39} (2011), no.~4,
  1544--1590.

\bibitem{bordenave-caputo-chafai-heavygirko}
\bysame, \emph{Spectrum of non-{H}ermitian heavy tailed random matrices},
  {C}ommunications in {M}athematical {P}hysics \textbf{307} (2011), no.~2,
  513--560.

\bibitem{MR2892961}
\bysame, \emph{Circular law theorem for random {M}arkov matrices}, Probab.
  Theory Related Fields \textbf{152} (2012), no.~3-4, 751--779. \MR{2892961}

\bibitem{bordenave-chafai-changchun}
Charles Bordenave and Djalil Chafa{\"{\i}}, \emph{Around the circular law},
  Probab. Surveys \textbf{9} (2012), no.~0, 1--89.

\bibitem{MR866489}
L.~G. Brown, \emph{Lidski\u\i's theorem in the type {${\rm II}$} case},
  Geometric methods in operator algebras (Kyoto, 1983), Pitman Res. Notes Math.
  Ser., vol. 123, Longman Sci. Tech., Harlow, 1986, pp.~1--35.

\bibitem{MR2206341}
W.~Bryc, A.~Dembo, and T.~Jiang, \emph{Spectral measure of large random
  {H}ankel, {M}arkov and {T}oeplitz matrices}, Ann. Probab. \textbf{34} (2006),
  no.~1, 1--38.

\bibitem{MCMC04}
M.~Capitaine and M.~Casalis, \emph{Asymptotic freeness by generalized moments
  for {G}aussian and {W}ishart matrices. {A}pplication to beta random
  matrices}, Indiana Univ. Math. J. \textbf{53} (2004), no.~2, 397--431.

\bibitem{djalil-nccl}
D.~Chafa{\"{\i}}, \emph{Circular law for noncentral random matrices}, Journal
  of Theoretical Probability \textbf{23} (2010), no.~4, 945--950.

\bibitem{conway90}
John~B. Conway, \emph{A course in functional analysis}, second ed., Graduate
  Texts in Mathematics, vol.~96, Springer-Verlag, New York, 1990.

\bibitem{MR2885424}
Colin Cooper and Alan Frieze, \emph{Stationary distribution and cover time of
  random walks on random digraphs}, J. Combin. Theory Ser. B \textbf{102}
  (2012), no.~2, 329--362. \MR{2885424}

\bibitem{MR0120167}
P.~Erd{\H{o}}s and A.~R{\'e}nyi, \emph{On random graphs. {I}}, Publ. Math.
  Debrecen \textbf{6} (1959), 290--297.

\bibitem{FZ97}
J.~Feinberg and A.~Zee, \emph{Non-{H}ermitian {R}andom {M}atrix {T}heory:
  {M}ethod of {H}ermitian {R}eduction}, Nucl. Phys. B (1997), no.~3, 579--608.

\bibitem{MR637828}
Z.~F{\"u}redi and J.~Koml{\'o}s, \emph{The eigenvalues of random symmetric
  matrices}, Combinatorica \textbf{1} (1981), no.~3, 233--241.

\bibitem{gotze-tikhomirov-new}
F.~G{\"o}tze and A.~Tikhomirov, \emph{The circular law for random matrices},
  Ann. Probab. \textbf{38} (2010), no.~4, 1444--1491. \MR{2663633}

\bibitem{Gudowska-Nowak}
E.~Gudowska-Nowak, A.~Jarosz, M.~Nowak, and G.~Pappe, \emph{Towards
  non-{H}ermitian random {L}\'evy matrices}, Acta Physica Polonica B
  \textbf{38} (2007), no.~13, 4089--4104.

\bibitem{MR2535081}
Adityanand Guntuboyina and Hannes Leeb, \emph{Concentration of the spectral
  measure of large {W}ishart matrices with dependent entries}, Electron.
  Commun. Probab. \textbf{14} (2009), 334--342. \MR{2535081 (2011c:60023)}

\bibitem{MR1784419}
U.~Haagerup and F.~Larsen, \emph{Brown's spectral distribution measure for
  {$R$}-diagonal elements in finite von {N}eumann algebras}, J. Funct. Anal.
  \textbf{176} (2000), no.~2, 331--367.

\bibitem{MR2339369}
Uffe Haagerup and Hanne Schultz, \emph{Brown measures of unbounded operators
  affiliated with a finite von {N}eumann algebra}, Math. Scand. \textbf{100}
  (2007), no.~2, 209--263. \MR{2339369 (2008m:46139)}

\bibitem{MR1084815}
R.~A. Horn and Ch.~R. Johnson, \emph{Matrix analysis}, Cambridge University
  Press, Cambridge, 1990, Corrected reprint of the 1985 original.

\bibitem{MR1288752}
\bysame, \emph{Topics in matrix analysis}, Cambridge University Press,
  Cambridge, 1994, Corrected reprint of the 1991 original.

\bibitem{MR1782847}
Svante Janson, Tomasz {\L}uczak, and Andrzej Rucinski, \emph{Random graphs},
  Wiley-Interscience Series in Discrete Mathematics and Optimization,
  Wiley-Interscience, New York, 2000.

\bibitem{MR0101545}
A.~Kolmogorov, \emph{Sur les propri\'et\'es des fonctions de concentrations de
  {M}. {P}. {L}\'evy}, Ann. Inst. H. Poincar\'e \textbf{16} (1958), 27--34.
  \MR{0101545 (21 \#355)}

\bibitem{MR2146352}
A.~E. Litvak, A.~Pajor, M.~Rudelson, and N.~Tomczak-Jaegermann, \emph{Smallest
  singular value of random matrices and geometry of random polytopes}, Adv.
  Math. \textbf{195} (2005), no.~2, 491--523.

\bibitem{reedsimon}
M.~Reed and B.~Simon, \emph{Methods of modern mathematical physics. {I}},
  second ed., Academic Press Inc. [Harcourt Brace Jovanovich Publishers], New
  York, 1980, Functional analysis.

\bibitem{rogers2010}
Tim Rogers, \emph{Universal sum and product rules for random matrices}, J.
  Math. Phys. \textbf{51} (2010), no.~9, 093304, 15.

\bibitem{MR0131894}
B.~A. Rogozin, \emph{On the increase of dispersion of sums of independent
  random variables.}, Teor. Verojatnost. i Primenen \textbf{6} (1961),
  106--108. \MR{0131894 (24 \#A1741)}

\bibitem{MR2407948}
M.~Rudelson and R.~Vershynin, \emph{The {L}ittlewood-{O}fford problem and
  invertibility of random matrices}, Adv. Math. \textbf{218} (2008), no.~2,
  600--633.

\bibitem{MR1284550}
J.~W. Silverstein, \emph{The spectral radii and norms of large-dimensional
  non-central random matrices}, Comm. Statist. Stochastic Models \textbf{10}
  (1994), no.~3, 525--532.

\bibitem{MR2409368}
T.~Tao and V.~Vu, \emph{Random matrices: the circular law}, Commun. Contemp.
  Math. \textbf{10} (2008), no.~2, 261--307.

\bibitem{tao-vu-cirlaw-bis}
\bysame, \emph{Random matrices: universality of {ESD}s and the circular law},
  Ann. Probab. \textbf{38} (2010), no.~5, 2023--2065, With an appendix by
  Manjunath Krishnapur. \MR{2722794}

\bibitem{MR3010398}
Terence Tao, \emph{Outliers in the spectrum of iid matrices with bounded rank
  perturbations}, Probab. Theory Related Fields \textbf{155} (2013), no.~1-2,
  231--263. \MR{3010398}

\bibitem{MR2507275}
Terence Tao and Van Vu, \emph{From the {L}ittlewood-{O}fford problem to the
  circular law: universality of the spectral distribution of random matrices},
  Bull. Amer. Math. Soc. (N.S.) \textbf{46} (2009), no.~3, 377--396.
  \MR{2507275 (2010b:15047)}

\bibitem{MR1217253}
D.~V. Voiculescu, K.~J. Dykema, and A.~Nica, \emph{Free random variables}, CRM
  Monograph Series, vol.~1, American Mathematical Society, Providence, RI,
  1992, A noncommutative probability approach to free products with
  applications to random matrices, operator algebras and harmonic analysis on
  free groups.

\bibitem{MR1744647}
Dan Voiculescu, \emph{The coalgebra of the free difference quotient and free
  probability}, Internat. Math. Res. Notices (2000), no.~2, 79--106.

\bibitem{MR2384414}
Van~H. Vu, \emph{Spectral norm of random matrices}, Combinatorica \textbf{27}
  (2007), no.~6, 721--736.

\bibitem{MR2977992}
Philip~Matchett Wood, \emph{Universality and the circular law for sparse random
  matrices}, Ann. Appl. Probab. \textbf{22} (2012), no.~3, 1266--1300.
  \MR{2977992}

\bibitem{MR950344}
Y.~Q. Yin, Z.~D. Bai, and P.~R. Krishnaiah, \emph{On the limit of the largest
  eigenvalue of the large-dimensional sample covariance matrix}, Probab. Theory
  Related Fields \textbf{78} (1988), no.~4, 509--521.

\end{thebibliography}

\providecommand{\bysame}{\leavevmode\hbox to3em{\hrulefill}\thinspace}
\providecommand{\MR}{\relax\ifhmode\unskip\space\fi MR }
\providecommand{\MRhref}[2]{%
  \href{http://www.ams.org/mathscinet-getitem?mr=#1}{#2}
}
\providecommand{\href}[2]{#2}

\end{document}